\documentclass{article}

\usepackage{graphicx,url}
\usepackage[all]{xy}
\pdfoutput=1
\newtheorem{theorem}{Theorem}
\newtheorem{lemma}[theorem]{Lemma}
\newtheorem{proposition}[theorem]{Proposition}
\newtheorem{corollary}[theorem]{Corollary}
\newtheorem{definition}{Definition}

\usepackage{amsfonts,amssymb,amsmath}
\usepackage{hyperref}
\usepackage{url}

\newcommand{\C}{\ensuremath{\mathbb{C}}}

\newcommand{\R}{\ensuremath{\mathbb{R}}}

\newcommand{\Z}{\ensuremath{\mathbb{Z}}}

\date{}
\begin{document}

\title{Non-avoided crossings for $n$-body balanced configurations in $\R^3$ near a central configuration}
\author{Alain Chenciner\\
  \\
  \small Observatoire de Paris, IMCCE (UMR 8028), ASD\\
 \small \texttt{chenciner@imcce.fr}\\ \small\&\\
  \small D\'epartement de math\'ematique, Universit\'e Paris VII}

\maketitle

\hangindent=3cm\hangafter =-10 \noindent{\it \`A Jacques Laskar, avec amiti\'e, admiration et  la joie\\ de poursuivre cette longue route ensemble}\medskip

\begin{abstract}

The  {\it balanced configurations} are those $n$-body configurations which admit a relative equilibrium motion in a Euclidean space $E$ of high enough dimension $2p$, (see \cite{AC,C2}). They are characterized by the commutation of two symmetric endomorphisms of the $(n-1)$-dimensional Euclidean space $\mathcal{D}^*$ of {\it codispositions}, the {\it intrinsic inertia endomorphism} $B$ which encodes the shape and the {\it Wintner-Conley endomorphism} $A$ which encodes the forces.
In general, $p$ is the dimension $d$ of the configuration, which is also the rank of $B$. Lowering to $2(d-1)$ the dimension of $E$ occurs when the restriction $A_B$ of $A$ to the (invariant) image of $B$ possesses a double eigenvalue. This condition is well known to be of codimension 2 in the space of all $d\times d$ symmetric endomorphisms (hence the {\it avoided crossings} of physicists). If $d=3$, the subset formed by the endomorphisms $A_B$ of  balanced configurations is of dimension $3$, and endomorphisms with a double eigenvalue should a priori form 1-dimensional families.  But, due to the homogeneity of the equations, this would mean that a (similarity class of) central configuration in $\R^3$ is in general  an isolated point in the set of balanced configurations of the same dimension which admit a relative equilibrium motion in $\R^4$. That this is not the case for the regular tetrahedron with very symmetric choices of masses was known (see \cite{C2}); I prove here that the same holds whatever be the four masses: {\it stemming from the regular tetrahedron, there are always (generically three) non-trivial families of 4-body balanced configurations which admit a relative equilibrium motion in $\R^4$.} For more bodies, the same result follows easily from the commutation of the endomorphisms $A$ and $B$, provided a certain property $(H)$ is satisfied (proposition \ref{d<n-1}); the end of the paper is a detailed study of the case of 4 bodies, with a special attention to the bifurcation locus in the frequency polytope (see \cite{C1,CJ,HZ}) of the regular tetrahedron with generic masses. It is fair to say that the search for a proof was provoked by the result, opposite to my first expectations, of a symbolic computation made at my request by Jacques Laskar (section \ref{Trip}). 
\end{abstract}
\goodbreak

\tableofcontents
\section{Central configurations, balanced configurations and their relative equilibria}

\noindent {\it Central configurations} are the $N$-body configurations which collapse homothetically on their center of mass when released without initial velocity; they are known since Euler and Lagrange to admit periodic homographic motions of all eccentricities and in particular periodic relative equilibrium motions. More generally, {\it balanced configurations} (see\cite{AC,C2}) are the $N$-body configurations which admit a (in general quasi-periodic) relative equilibrium motion in a Euclidean space of high enough dimension. They are characterized by the commutation of two endomorphisms of the {\it codisposition space} ${\cal D}^*$, the first one $A$ which characterizes the attraction forces between the bodies of the configuration, the second one $B$, an intrinsic inertia which encodes the shape of the configuration.
\goodbreak

\subsection{From $B$ to $A$: shapes and forces}\label{BtoA}
An $n$-body configuration $x=(\vec r_1,\cdots, \vec r_n)$ up to translation in the Euclidean space\footnote{The Euclidean structure is identified with an isomorphism $\epsilon:E\to E^*$ from $E$ to its dual.} $(E,\epsilon)$ is a mapping 
$$x:{\cal D}^*\to E,\; (\xi_1,\cdots,\xi_n)\mapsto \sum_{i=1}^n{\xi_i\vec r_i},$$ 
equivalently an element of ${\cal D}\otimes E$, where $${\cal D}:=\R^n/(1,\cdots,1)\R$$ 
is the {\it dispositions} space and ${\cal D}^*=\{(\xi_1,\cdots,\xi_n)| \sum_{i=1}^n\xi_i=0\}$ is its dual.
Fixing the masses naturally endows ${\mathcal D}$ (resp. ${\mathcal D}^*$) with the {\it mass Euclidean structure} $\mu:{\cal D}\to{\cal D}^*$ defined by
$$\mu(x_1,\ldots,x_n)=\bigl(m_1(x_1-x_G),\dots,m_n(x_n-x_G)\bigr),$$
where $x_G=(m_1 x_1+\cdots+m_n
x_n)/\sum{m_i}$ is the center of mass of the $x_i$
(resp. $\mu^{-1}:{\cal D}^*\to{\cal D}$ defined by $\mu^{-1}(\xi_1,\dots,\xi_n)=\bigl(\frac{\xi_1}{m_1},\dots,\frac{\xi_n}{m_n}\bigr)
$).
\smallskip

\noindent The {\it intrinsic inertia form}  (resp. the {\it (dual) inertia form}) of the configuration $x$ is the quadratic form $\beta$ on ${\cal D}^*$ (resp. the quadratic form $b$ on $E^*$), defined by
\begin{equation*} 
\begin{split}
&\beta=x^*\epsilon=x^{tr}\circ\epsilon\circ x\in Hom_s({\mathcal D}^*,{\mathcal D})\equiv Q({\mathcal D}^*)\equiv{\mathcal D}\odot{\mathcal D},\\
\hbox{resp.}\quad &b=(x^{tr})^*\mu=x\circ\mu\circ x^{tr}\in Hom_s(E^*,E)\equiv Q(E^*)\equiv E\odot E.
\end{split}
\end{equation*}
The form $\beta$ defines the configuration up to a rigid motion
(translation and rotation) in $E$. The {\it intrinsic inertia endomorphism} (resp. the {\it (dual) inertia endomorphism}) are respectively the $\mu^{-1}$-symmetric (resp. $\epsilon$-symmetric) endomorphisms
$$B=\mu\circ\beta:{\cal D}^*\to{\cal D}^*\quad(\hbox{resp.}\quad S=b\circ\epsilon:E\to E).$$
Let $$U(x)=\sum_{i<j}m_im_j\Phi(r_{ij}^2),\quad\hbox{where}\quad \Phi(s)={\mathcal G}s^{-\frac{1}{2}},$$
be the potential\footnote{Nothing essential would change for a more general homogeneous $\Phi$.} of the $n$-body configuration $x$. Its invariance under isometries implies the factorization $U(x)=\hat U(\beta)=\tilde U(B)$. The Wintner-Conley endomorphism associated to $x_0$ is the ($\mu^{-1}-symmetric$) endomorphism of ${\mathcal D}^*$ defined by 
$$A=d\hat U(\beta)\circ\mu^{-1}=\mu\circ d\tilde U(B)\circ \mu^{-1}.$$
It is characterized (see \cite{AC,C2}) by the fact that  the equations of motion are $$\ddot x=2x\circ A.$$ 
\medskip

\noindent In some $\mu^{-1}$-orthonormal basis, the two $\mu^{-1}$-symmetric endomorphisms $A$ and $B$ of ${\mathcal D}^*$, are represented by  symmetric matrices. 
Giving any one of these two matrices is equivalent to giving the squared mutual distances $s_{ij}=r_{ij}^2$, that is defining the configuration up to isometry (see \cite{AC}), hence the mapping $F:B\mapsto A$ is well defined outside of the collisions (that is when all the $r_{ij}$ are strictly positive)  and bijective on its image: {\sl the shape of the configuration determines the forces and the forces determine the shape.}  Moreover,
 \begin{lemma}\label{BA} The mapping $F$ is a diffeomorphism.  
\end{lemma}
Indeed, in well chosen bases of the space of symmetric matrices, the coordinates of $B$ are the squared mutual distances $s_{ij}=r_{ij}^2,\; 1\le i<j<n$ while the coordinates of $A$ are the $\varphi(s_{ij}):=\Phi'(s_{ij})=-\frac{1}{2}\mathcal{G}s_{ij}^{-\frac{3}{2}}$.

\subsection{Balanced configurations and their relative equilibria}\label{Balcon}
\noindent A {\it rigid motion} is a solution of the equations of motion along which the mutual distances $r_{ij}$ remain constant. Such a motion is necessarily a {\it relative equilibrium}, that is an equilibrium of the equations after reduction of their natural symmetry under isometries (\cite{AC} Proposition 2.5). Moreover, relative equilibria are of the following form (\cite{AC} Propositions 2.8 and 2.9):
$$x(t)=e^{\Omega t}x,$$
where $\Omega$ is a constant $\epsilon$-antisymmetric\footnote{i.e. such that $\varpi=\epsilon\circ\Omega=\Omega^{tr}\circ\epsilon=\varpi^{tr}\in\Lambda^2E^*$.} isomorphism\footnote{It is assumed that $E$ is the space effectively visited by the motion.} of the (necessarily even dimensional) Euclidean space $(E,\epsilon)$ and $x$ belongs to the very special class of {\it balanced configurations} which we now characterize: 
it follows from the equations of motion that
$$\Omega^2\circ x=2x\circ A.$$

\noindent {\it From now on, we shall identify ${\mathcal D}$ and $E$ with their respective duals ${\mathcal D}^*$ and $E^*$ using their Euclidean structures $\mu$ and $\epsilon$. Moreover, we shall choose a $\mu^{-1}$-orthonormal basis of ${\cal D}^*$ and an $\epsilon$-orthonormal basis of $E$
and represent endomorphisms of ${\cal D}^*$ and endomorphisms of $E$ by matrices in such bases.}
\smallskip

\noindent If $X$ is the $2p\times (n-1)$ matrix representing $x_0$ in these bases, we have:
$$S=XX^{tr}, \quad B=X^{tr} X, \quad \Omega^2X=2XA.$$
From the symmetry of the matrices $2BA=X^{tr}\Omega^2X$ and $\Omega^2S=2XAX^{tr}$, we deduce the vanishing of the following commutators:
$$[A,B]=0,\quad [\Omega^2,S]=0.$$
\goodbreak

\noindent Of course, as $\Omega^2$ commutes with $e^{\Omega t}$, one can replace $S$ by the inertia endomorphism $S(t)=X(t)X(t)^{tr}$.
The first equation, independent of the dimension of the ambient space, was shown in \cite{AC} to define the balanced configurations. 
The following diagram summarizes these relations (on the left, the ``side of the bodies", on the right the ``side of ambient space"):
\begin{center}
\includegraphics[scale=0.45]{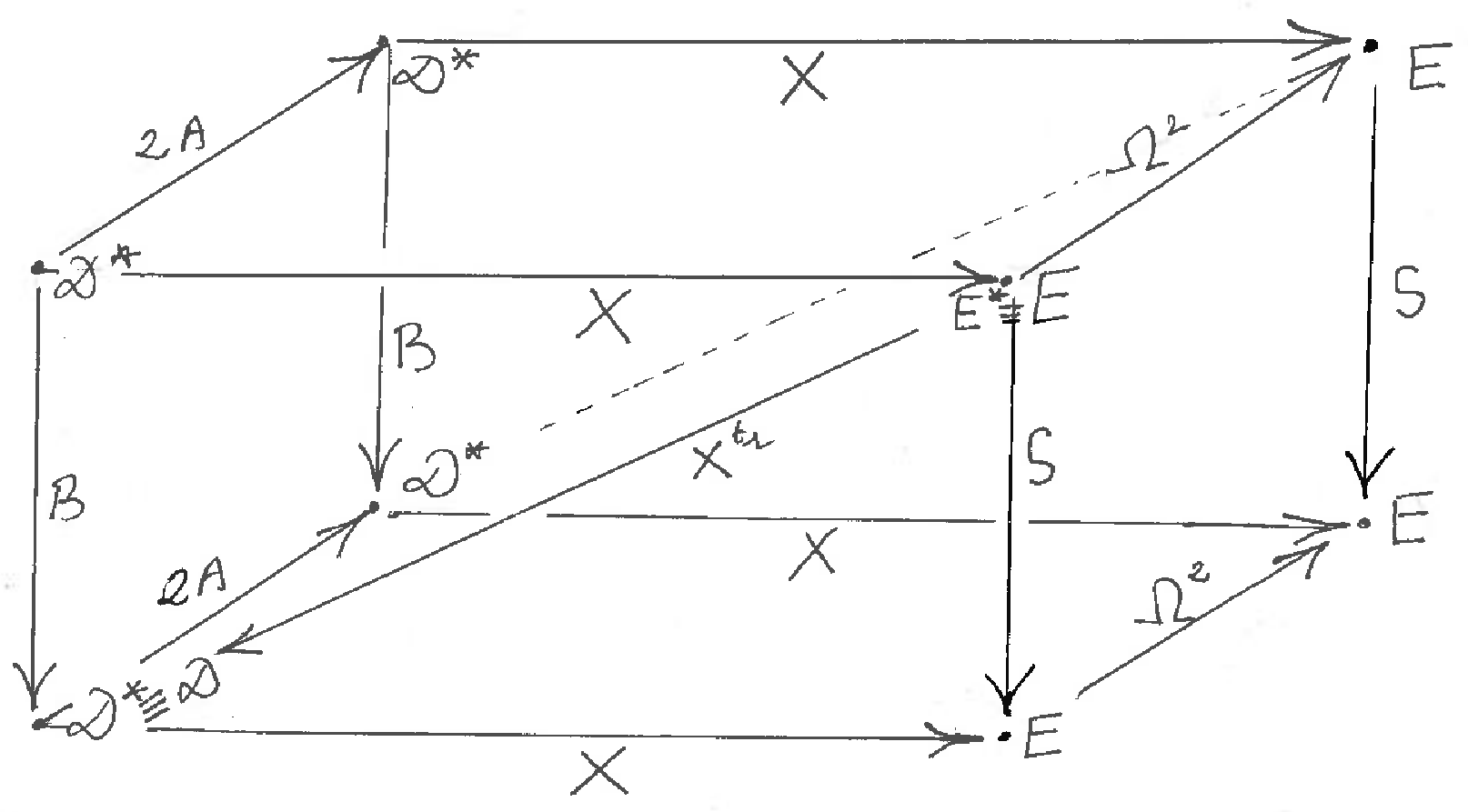} 
\end{center}
\noindent In terms of mutual distances, the equations of balanced configurations are (\cite{AC})
$$P_{ijk}=-\frac{1}{2}\nabla_{ijk}+\frac{1}{2}\sum_{l\not=i,j,k}Y_{ijk}^l=0,\; i<j<k,$$
with (recalling that $\varphi(s)=-\frac{1}{2}s^{-\frac{3}{2}}$)
\begin{equation*}
\begin{split}
\nabla_{ijk}&=\begin{vmatrix}
1&1&1\\
m_i(r_{jk}^2-r_{ki}^2-r_{ij}^2)&m_j(r_{ki}^2-r_{ij}^2-r_{jk}^2)&m_k(r_{ij}^2-r_{jk}^2-r_{ki}^2)\\
\varphi(r_{jk}^{2})&\varphi(r_{ki}^{2})&\varphi(r_{ij}^{2})
\end{vmatrix},\\
Y_{ijk}^l&=
m_l\begin{vmatrix}
1&1&1\\
r_{jk}^2+r_{il}^2&r_{ki}^2+r_{jl}^2&r_{ij}^2+r_{kl}^2\\
\varphi(r_{il}^{2})&\varphi(r_{jl}^{2})&\varphi(r_{kl}^{2})
\end{vmatrix}.
\end{split}
\end{equation*}
As soon as the number $n$ of bodies is greater than three, these equations are not independent:
they are obtained by taking the exterior product by $(1,1,\cdots,1)$ of the equation $[A,B]=0$, that is by embedding the space $\wedge^2\mathcal D$, whose dimension is ${n\choose2}=\frac{(n-1)(n-2)}{2}$, into the space $\wedge^3\R^n$, whose dimension is ${n\choose3}=\frac{n(n-1)(n-2)}{6}$.
\smallskip

\noindent Finally, let us recall (see \cite{AC}) that the balanced (resp. central) configurations are those configurations whose intinsic inertia endomorphism $B$ is a critical point of the potential function $\tilde U$ restricted to its {\it isospectral submanifold}, consisting in all the $\mu^{-1}$-symmetric endomorphisms with the same spectrum (resp. restricted to the submanifold of symmetric $\mu^{-1}$-symmetric endomorphisms which have the same trace $I$).

\subsection{The invariant subspaces of $\Omega$ and the minimal dimension of $E$}\label{dimMin}
Each $n$-body balanced configuration $x$ admits relative equilibrium motions in a Euclidean space $E$ of high enough even dimension (twice the rank of $x$, that is twice the rank of $B$, suffices, see \cite{AC} Proposition 2.8). The smallest possible dimension allowing such a motion depends on the  multiplicities of the eigenvalues of the Wintner-Conley endomorphism $A$ (\cite{AC} Remark 2.11). In order to explain this, we study the invariant subspaces of the instantaneous rotation matrix $\Omega$. 
\vskip0.1cm
 
\noindent Fixing the balanced configuration $x$, we choose a $\mu^{-1}$-orthonormal basis of ${\cal D}^*$ and an $\epsilon$-orthonormal basis  of $E$ such that the matrices $B$ and $S$ representing the  inertia endomorphisms of $x$ be diagonal:

\begin{equation*}
\left\{
\begin{split}
B&=diag(b_1,\cdots,b_{n-1}),\\
S&=diag(\sigma_1,\cdots,\sigma_{2p}).
\end{split}\right.
\end{equation*}
From the above mentioned  commutations, it follows that such bases can be chosen so as to satisfy also
\begin{equation*}
\left\{
\begin{split}
2A&=diag(-\lambda_1,\cdots,-\lambda_{n-1}),\\
\Omega^2&=diag(-\omega^2_1,\cdots,-\omega^2_{2p}).
\end{split}
\right.
\end{equation*}
Note that all the eigenvalues of $2A$ are strictly negative because the Newton force is attractive.
The non-zero eigenvalues of $S=XX^{tr}$ and $B=X^{tr} X$ are the same; their number is the dimension of the configuration, that is the rank $d=\hbox{dim\,Im}X$ of $X$.
Moreover, as $\hbox{Im}\,X=\hbox{Im}\,S$ is generated by vectors of the basis of $E$,  we can suppose, after a possible reordering of the bases of ${\cal D}^*$ and $E$,  that $X$ is of the form
$
X=\begin{pmatrix}
V&W\\
0&0
\end{pmatrix},
$
where the $d\times d$ upper left block $V$ is invertible.
Moreover, as $B=X^{tr} X$ is diagonal, $W=0$, that is 
$$
X=\begin{pmatrix}
V&0\\
0&0
\end{pmatrix},\quad \hbox{with}\quad V:\hbox{Im}B\to\hbox{Im}S=\hbox{Im}X=\hbox{Im}x\quad \hbox{an isomorphism},
$$
and $\;VV^{tr}=diag(\sigma_1,\cdots,\sigma_{d}),\quad V^{tr}V=diag(b_1,\cdots,b_{d})$, hence
$$diag(\sigma_1,\cdots,\sigma_{d})V=Vdiag(b_1,\cdots,b_{d}).$$
Finally, the equation $\Omega^2X=2XA$, is equivalent to
$$diag(-\omega^2_1,\cdots,-\omega^2_{d})V=Vdiag(-\lambda_1,\cdots,-\lambda_{d}),$$ 
hence, after possibly replacing $V$ by its product $VP$ with a permutation matrix, which amounts to permuting the first $d$ vectors of the basis of $\mathcal{D}^*$, which generate $\hbox{Im}B$, one can suppose that 
$$\omega_k^2=\lambda_k,\; \hbox{for}\; k=1,\cdots, d,$$
while the $b_i,\; i=1,\cdots,d,$ are a permutation of the $\sigma_i,\; i=1,\cdots,d$.
\smallskip

\noindent This shows in particular, and this comes as no surprise,  that it is only the restriction $A_B$ of $A$ to the image of $B$ which plays a role. Recall that a necessary and sufficient condition for $x_0$ to be a central configuration is that the restriction of $A$ to this subspace be proportional to the Identity, that is $\lambda_1=\lambda_2=\cdots=\lambda_d$.
\goodbreak

\noindent{\bf Notation.} {\it Bases of ${\cal D}^*$ and $E$ with the properties above will be denoted respectively
$\{u_1,\cdots,u_{n-1}\}$ and $\{\rho_1,\cdots,\rho_{2p}\}$. In particular, $u_1,\cdots,u_d$ generate $\hbox{Im}B$ and $\rho_1,\cdots,\rho_d$ generate $\hbox{Im}\,x=\hbox{Im}\,S$.}
\medskip

\noindent The real invariant planes of  $\Omega$ can be generated either by the couple formed by a vector $\rho_k\in\hbox{Im}\,x$ and a vector $\rho_{d+l}$ in the orthogonal $(\hbox{Im}\,x)^\perp$, or by the couple formed by two vectors  $\rho_k,\rho_l\in\hbox{Im}\,x$, both associated with the same eigenvalue $\lambda_k=\lambda_l$ of $A$. Similar descriptions hold for higher dimensional invariant subspaces. It follows that,
writing $\tilde\omega_1^2,\cdots,\tilde\omega_r^2$ the distinct values taken by the $\omega_k^2=\lambda_k, k=1,\cdots,d$, a space $E$ of minimal dimension where a relative equilibrium motion with such a configuration may take place decomposes into a direct sum $E_1\oplus\cdots\oplus E_r$ of eigenspaces of $\Omega$, which are complex\footnote{more precisely ``hermitian", that is such that each $J_l$ is an isometry.} spaces $(E_l,J_l)$, and the motion is quasi-periodic of the form $$x(t)=\bigl(x_1(t),\cdots,x_r(t)\bigr),\quad \hbox{with}\quad x_l(t)=e^{\tilde\omega_lt}{x}_l,\; l=1,\cdots,r .$$
When $r=1$, that is when $\lambda_1=\cdots=\lambda_d=\tilde\omega^2$, which means that the configuration $x$ is central, the motion becomes periodic, of the form $x(t)=e^{\tilde\omega Jt}x$, with $J$ a complex (hermitian) structure on $E$. 
\smallskip

\noindent Let us denote by $$\varpi=\epsilon\circ\Omega\in\wedge^2E^*$$ 
the instantaneous rotation bivector.  
The main possibilities of invariant spaces for $\Omega$ can be read on the inverse image of $\varpi$ by $x$:
$$x^*\varpi=x^{tr}\circ\varpi\circ x=-\rho,$$
which is nothing (up to sign) but the antisymmetric part of $x^{tr}\circ\epsilon\circ y$:
$$\rho=\frac{1}{2}\left(-x^{tr}\circ\epsilon\circ y+y^{tr}\circ\epsilon\circ x\right),$$
which was introduced in \cite{La} in the case of 3 bodies and in \cite{AC} in the general case of $n$ bodies.   In term of matrices,
the endomorphism $R=\mu\circ\rho$ of ${\mathcal D}^*$ is represented by
 $$R=\frac{1}{2}\left(-X^{tr}Y+Y^{tr}X\right)=-X^{tr}\Omega X.$$
 Indeed, only contribute to $\rho$ the invariant subspaces of $\Omega$ entirely contained in $\hbox{Im}\,x_0$. In particular, $\rho$ vanishes in the generic case where the eigenvalues $\lambda_i$ of $A|_{\hbox{\small Im} B}$ are all distinct, that is when the motion takes place in a space of dimension twice the one of $\hbox{Im}\,x$  (compare \cite{AC} Remark 2.11).  
\smallskip

\subsection{A criterion for degeneracy}\label{deg} The equalities $\Omega^2X=2XA$ and $R=-X^{tr}\Omega X$ imply the commutation of $R$ with the Wintner-Conley matrix: $[A,R]=0$ (compare to  the equations of relative equilibrium in \cite{AC}). This looks quite natural in view of the following lemma, where ``degenerate" means ``possess a multiple eigenvalue":
\begin{lemma}[Lax \cite{L}]\label{Lax} A real symmetric matrix $A$ is degenerate if and only if it commutes with some nonzero real antisymmetric matrix $R$.
\end{lemma}
The proof is obvious in an orthonormal basis where $A$ is diagonal.
\goodbreak

\noindent Of course, the existence of a double eigenvalue is also equivalent to the vanishing of the {\it discriminant}, that is the resultant of the characteristic polynomial and its derivative. But already for $3\times 3$ symmetric matrices, the discriminant is a quite long homogeneous degree six polynomial in the 6 coefficients of the matrix, which can be written as a sum of 5 squares see \cite{D}); 
and  it is not even clear on this expression that, as was already known to Von Neuman and Wigner, its regular part (corresponding to the existence of exactly one pair of equal eigenvalues) defines a codimension 2 submanifold (this is the classical phenomenon of {\it avoided crossings} in quantum mechanics, the obvious geometric proof of which can be found in \cite{A}).  We shall use this criterion when studying the degeneracies of $B$ and $A$ in the case of 4 bodies (see sections  \ref{EquDeg} and \ref{Trip}).
.

\subsection{Central configurations of general type}\label{NonDeg} 

\begin{definition}\label{GenType} The balanced configuration $x_0$ is said to be {\rm of general type} if

1) the non-zero eigenvalues of its intrinsic inertia endomorphism $B_0$ are all distinct;

2) $B_0$ is a non-degenerate critical point of the restriction of the potential $\tilde U$ to its isospectral submanifold. 
\end{definition}
\begin{lemma}\label{param}
The intrinsic inertia endomorphisms $B$ of balanced configurations $x$ close enough to a balanced configuration $x_0$ of general type and of the same rank $d$, form a $d$-dimensional submanifold of the space of $(n-1)\times(n-1)$-symmetric endomorphisms of rank $d$, which intersects transversally the isospectral submanifold of $B_0$. 
\end{lemma}
\begin{proof}
Isospectral submanifolds of endomorphisms $B$ close to $B_0$ and with the same rank $d$ all have the same codimension $d$ as they may be labelled by their non-zero (and distinct) eigenvalues. As a non-degenerate critical point is isolated and differentiably stable under perturbations, the lemma follows : the non-zero eigenvalues of $B$ define coordinates in the set formed by the inertia matrices $B$ of balanced configurations close to $B_0$. 
\end{proof}





\subsubsection{The case of maximal rank ($d=n-1$)}
In this case, one can replace $B$ by $A$ in the parametrization given by lemma \ref{param} and this implies that, for generic balanced configurations, the degeneracy of $A$ becomes a codimension 1 property (hence the "non avoided crossings"):
\begin{proposition}\label{PropGenType} If $x_0$ is an $n$-body central configuration of general type of rank $n-1$, the intrinsic inertia endomorphisms $B$ of the balanced configurations $x$ close enough to $x_0$ form an $(n-1)$-dimensional submanifold, which can be parametrized either by their eigenvalues $(b_1,\cdots, b_{n-1})$ or by the eigenvalues [up to the factor -1/2] $(\lambda_1,\cdots,\lambda_{n-1})$ of their Wintner-Conley endomorphisms $A$.
\end{proposition} 

\begin{proof}The first part is nothing but lemma \ref{param}.  
For the second part, we represent symmetric endomorphisms of ${\mathcal D}^*$ by symmetric matrices in the unique (up to permutation) 
$\mu^{-1}$-orthonormal basis of ${\mathcal D}^*$ which diagonalizes $B_0$, and hence also diagonalizes $A_0=A(B_0)$; given any balanced configuration $B$ close enough to $B_0$, there exists a unique rotation $R=R(B)\in O({\mathcal D}^*)$ such that $RBR^{-1}$ (and hence also 
$RAR^{-1}$) is diagonal. The conclusion follows from the 
\begin{lemma}
Under the hypotheses of Proposition \ref{PropGenType}, the map ${\mathcal A}$ from balanced configurations $B$ close enough to $B_0$ to diagonal matrices defined by
$$\mathcal A(B)= {R(B)}A(B)R(B)^{-1}=-\frac{1}{2}diag(\lambda_1,\cdots,\lambda_{n-1}),$$ 
 is a diffeomorphism onto its image.
\end{lemma}
\begin{proof}
The derivative of ${\mathcal A}$ at the central configuration $B_0$ is
$$d{\mathcal A}(B_0)\Delta B=dA(B_0)\Delta B+\bigl[dR(B_0)\Delta B,A_0\bigr]=dA(B_0)\Delta B,$$
because $A_0$ is proportional to the Identity. The conclusion follows because the map $B\mapsto A(B)$ is a diffeomorphism.
\end{proof}
\end{proof}
\goodbreak

\subsubsection{The general case}\label{general}
If the rank $d$ of the configuration is strictly smaller than $n-1$, it is natural to look at the the balanced configurations $x$ close enough to $x_0$ {\it with the same rank $d$}. For such configurations,  the analogue of Proposition \ref{PropGenType} now requires that the following condition (automatic if the rank of $x_0$ is $n-1$) be satisfied by $x_0$:
{\begin{quotation}
\noindent {\bf (H)}\hskip0.15cm {\it The balanced configuration $x_0$ (or its inertia $B_0$) is said to satisfy condition (H) if  the mapping $B\mapsto A|_{Im B},$
which to the intrinsic inertia of a balanced configuration $x$ of the same rank $d$ associates the restriction to its image of the Wintner-Conley endomorphism, is a local diffeomorphism at $B_0$.}
\end{quotation}}
\begin{proposition}\label{d<n-1}
If $x_0$ is an $n$-body central configuration of general type of rank $d$ satisfying {\bf (H)}, the intrinsic inertia endomorphisms $B$ of the balanced configurations $x$ close enough to $x_0$ and of the same rank $d$ form a $d$-dimensional submanifold, which can be parametrized either by their non-zero eigenvalues $(b_1,\cdots, b_{d})$ or by the eigenvalues [up to the factor -1/2] $(\lambda_1,\cdots,\lambda_{d})$ of the restriction $A|_{Im B}$ to their image of their Wintner-Conley endomorphism. In particular, for generic balanced configurations satisfying $(H)$, the degeneracy of $A$ is a codimension 1 property:
\end{proposition}
\begin{proof}
If $B$ is close to $B_0$, the unique (up to permutation) isometry of $Im B$ onto $Im B_0$ sending an eigenbasis of $B|_{Im B}$ onto an eigenbasis of $B_0|_{Im B_0}$ can be extended in a smooth way into an isometry $R(B)$ of ${\mathcal D}^*$ close to Identity.  In the eigenbasis of $B_0|_{Im B_0}$, we have
$$RBR^{-1}|_{Im B_0}=diag(b_1,\cdots,b_{d}),\quad \hbox{and}
\quad RAR^{-1}|_{Im B_0}=-\frac{1}{2}diag(\lambda_1,\cdots,\lambda_{d}).$$
Now, let $\mathcal A$ be as above the map from the set of balanced configurations of rank $d$ near $B_0$ to the set of diagonal matrices, defined by:
$${\mathcal A}(B)=R(B)A(B)R(B)^{-1}=-\frac{1}{2}diag(\lambda_1,\cdots,\lambda_{d}, D),$$
where $D$ is an $(n-1-d)\times(n-1-d)$ matrix.
The difference with the case of maximal rank is that, $A_0$ being only proportional to the Identity in restriction to the image of $B_0$,  the second term  $\bigl[dR(B_0)\Delta B,A_0\bigr]$ does not vanish. More precisely, If we write 
$$dR(B_0)\Delta B=\Delta R=
\begin{pmatrix}a&b\\ -b^{tr}&d
\end{pmatrix},\quad\hbox{and}\quad 
A_0=\begin{pmatrix}
\lambda_0Id&0\\
0&D_0
\end{pmatrix}$$
we have
$$\bigl[\Delta R,A_0\bigr]=
\begin{pmatrix}
0&b(D_0-\lambda_0Id)\\
(D_0-\lambda_0Id)b^{tr}&[d,D_0]
\end{pmatrix}.
$$
Hence, in restriction to $Im B_0$, the derivative of $\mathcal A$ at the central configuration $B_0$ again reduces to the derivative at $B_0$ of $B\mapsto A|_{Im B}$.

\end{proof}
\medskip

\noindent {\bf Remarks. 1)} Property {\bf (H)}  can be checked to hold in the (much too) simple case of $\Z/2\Z$-symmetric balanced configurations of rank 2 of two pairs of equal masses in the neighborhood of a generic (i.e. if some explicit condition $K\not=0$ holds) planar rhombus central configuration; moreover, if the two masses are distinct and their ratio avoids one single value $\gamma$, close to 0.575, or its inverse, the inertia ellipsoid of the central configuration is generic; see section \ref{2d}). 

{\bf 2)} Let us suppose that $x_0$ is a balanced configuration of general type of rank $d=n-2$. As there is only one zero eigenvalue in the spectrum of the inertia endomorphism $B_0$, the dimension of its isospectral manifold is the same as the one of the nearby isospectral manifolds of endomorphisms $B$ of maximal rank $n-1$. One deduces that $B_0$ is a regular point of the boundary of the set of inertia matrices of balanced configurations, the local equation of this boundary being the equality to zero of the last eigenvalue $b_{n-1}=\sigma_{n-1}$ of $B$. In terms of the eigenvalues of $A$, if the property {\bf (H)} is satisfied, this equation reads $\lambda_{n-1}=f(\lambda_1,\cdots,\lambda_{n-2})$. The (too simple) example of colinear configurations of 3 bodies is illustrated in the figure 2 of \cite{AC}; More significant examples are 4 bodies in $\R^2$ (see section \ref{2d}) or 5 bodies in $\R^3$.
\smallskip

%

\section{Bifurcations of a periodic relative equilibrium into a family of quasi-periodic ones}
We consider continuous families of quasi-periodic relative equilibria 
$$x_s(t)=e^{\Omega_st}x_s(0),\quad s\ge 0\; \hbox{small},$$
of a family $s\mapsto x_s(0)$ of balanced configurations in some Euclidean space $E$, originating from a periodic relative equilibrium of a central configuration $x_0$ (in particular, $\Omega_0=\omega J$, where $J$ is a complex structure on the ambient space $E$).  {\it We suppose that all the $x_s,\, s\ge 0$ have the same dimension $d$}. 
We make the following assumptions
(the second one can be satisfied by composing with a well chosen family of rotations): 

1) the spectral type of $\Omega_s$ is constant for $s>0$ small, 

2) the eigenspaces $E_1,\cdots, E_r$ of $\Omega_s$ have a limit when  $s\to 0$.
\smallskip

\noindent We shall study the two extreme cases: the ``generic" case where the dimension $2p$ of $E$ is twice the dimension $d$ of the configuration $x_0$ and the case where $2p=d$ if $d$ is even, $2p=d+1$ if $d$ is odd.
\goodbreak

\subsection{The generic case}\label{generic} 
If for $s>0$ small the first $d$ eigenvalues of the Wintner-Conley matrices $A_s$ of the balanced configuration $x_s(0)$ are all distinct, that is if the quasi-periodic relative equilibria $x_s(t)$ which bifurcate from the periodic relative equilibrium $x_0(t)$ have $d$ frequencies,
the eigenspaces of $\Omega_s$ are necessarily generated by an eigenvector of $\Omega_s^2$ (and hence of $S_s$) contained in $\hbox{Im}\,x_s$ and a vector orthogonal to  $\hbox{Im}\,x_s$. Going to the limit when $s$ tends to 0, we get that $\hbox{dim}\,E=2\,\hbox{dim}\,\hbox{Im}\,x_0$ and that the complex structure $J$ sends each eigenvector $\rho_k$ of $S_0$ onto a vector $k$ orthogonal to $\hbox{Im}\,x_0=\hbox{Im}\,S_0$. 
\smallskip

\noindent Given a central configuration of dimension $d$ in an euclidean space of dimension $2d$, a {\it basic hermitian structure} is one for which there exists a partition of an eigenbasis
of $S_0$ into $d$ pairs such that the planes generated by the two members of each pair are complex lines (see \cite{C1}).

\begin{definition}
In the case when $dim\,E= 2\,dim\,Im\,x_0$, a basic hermitian structure will be called ``of extrinsic type" if the partition is, as above,  of the type $\{\rho_1,v_{1}\}\cup\cdots\cup\{\rho_d,v_{d}\}$, where $\{\rho_1,\cdots,\rho_d\}$ is a basis of $Im\,x_0$ formed of eigenvectors of $S_0$ and $\{v_{1},\cdots,v_{d}\}$ is a basis of $Im\,x_0^\perp$.
\end{definition}

\noindent From the above discussion, we get

\begin{proposition}\label{extrinsic} Let $x(t)=e^{\omega  Jt}x_0$ be a relative equilibrium motion of a central configuration $x_0$ of dimension $d$ in a space $E$ of dimension $2d$. One supposes that, from this periodic relative equilibrium, stems a one parameter family of quasi-periodic relative equilibria with $d$ frequencies $x_s(t)=e^{\Omega_s t}x_{0},\, \Omega_0=\omega J$. Then the hermitian structure defined by $J$ on the Euclidean space $E$ is basic and of extrinsic type. 
\end{proposition}

\subsection{Bifurcations without increase of dimension}
Let $x_0\in\mathcal{D}\otimes E$ be a central configuration of dimension $d$ of $n$ bodies in the Euclidean space $E$ of dimension $2p=d$ if $d$ is even, $2p=d+1$ if $d$ is odd.
Let $x(t)=e^{\omega Jt}x_0$ be a relative equilibrium of $x_0$ in $E$ directed by the complex structure $J$.
From paragraph \ref{dimMin}, we know that a family of quasi-periodic relative equilibria of balanced configurations can bifurcate {\it in the same space} $E$ from this periodic relative equilibrium if and  only in two conditions are satisfied:
\vskip0.2cm

1) $E$ admits a direct sum decomposition $E=E_1\oplus\cdots\oplus E_r$ into at least two $J$-complex subspaces generated (over $\R$) by eigenvectors of the inertia $S_0$; 
\vskip0.2cm

2) The non-zero eigenvalues of the Wintner-Conley matrix of $x_0$ corresponding to eigenvectors generating any of these subspaces are equal.   
\smallskip

\noindent Thanks to proposition \ref{d<n-1}, one deduces the codimension of the corresponding bifurcations; in particular : \goodbreak

\begin{proposition}\label{NonGen}
Let $x_0$ be a central configuration of general type and satisfying {\bf(H)}, of dimension $d=2p-1$ or $d=2p$ in the Euclidean space $E$ of dimension $2p$.  In the manifold\footnote{with boundary if $d=2p-1$ and $n\ge 2p+1.$} of balanced configurations close enough to $x_0$ and of the same dimension, those which admit a relative equilibrium motion in $E$  form a stratified subset of dimension $p$. The main stratum, of dimension $p$, corresponds to quasi-periodic relative equilibria in $E$ with $p$ frequencies, the smallest stratum (if $p\ge 2$), of dimension 2, to quasi-periodic relative equilibria in $E$ with $2$ frequencies. 
\end{proposition} 
\noindent {\sl One should not forget that because of the projective invariance, the pertinent dimensions are respectively $p-1$ and $1$.}
\smallskip

\begin{proof} It is enough to notice that the main stratum corresponds to identities of the following type between the eigenvalues $\lambda_i$ of the endomorphism $A_B$ (see proposition \ref{PropGenType}):

\begin{equation*}
\begin{split}
&\lambda_{i_1}=\lambda_{j_1},\cdots \lambda_{i_{p-1}}=\lambda_{j_{p-1}}\quad\hbox{for $(2p-1)$-dimensional configurations},\\
&\lambda_{i_1}=\lambda_{j_1},\cdots \lambda_{i_{p}}=\lambda_{j_{p}}\quad\hbox{for $2p$-dimensional configurations}.
\end{split}
\end{equation*}
In both cases, the dimension of the stratum is
$$(2p-1)-(p-1)=2p-p=p.$$
\noindent In the same way, the smallest stratum corresponds to the equations
\begin{equation*}
\begin{split}
&\lambda_{i_1}=\cdots=\lambda_{i_{2p-2}}\quad\hbox{for $(2p-1)$-dimensional configurations},\\
&\lambda_{i_1}=\cdots=\lambda_{i_{2p-2}}\;\hbox{and}\; \lambda_{j_1}=\lambda_{j_2}\quad\hbox{for $2p$-dimensional configurations}.
\end{split}
\end{equation*}
In both cases, the dimension of the stratum is $$(2p-1)-(2p-3)=2p-(2p-3)-1=2.$$
\end{proof}
\noindent I leave to the reader the pleasure of describing the intermediate strata.
\medskip

\noindent {\bf Examples} 1) Central or balanced configurations of 3 bodies which are not of general type can be observed in figure 2 of \cite{AC}. On the other hand, one notices also on these figures examples of balanced configurations for which $B_0$ is proportional to the Identity but is nevertheless a regular point of the set of balanced configurations (this case is realized when the center of mass of the configuration coincides with the orthocentre of the triangle (see also ex. 3)).
\vskip0.1cm

2) Three-dimensional 4-body central configurations of general type are characterized in Corollary \ref{4bodyGenType}.
\vskip0.1cm

3) The three-dimensional balanced configurations of four bodies farthest from being of general type -- the ones with $B$ proportional to Identity -- are the orthocentric tetrahedra (i.e. the tetrahedra which possess an orthocenter, which is the intersection of the four heights) such that the orthocenter coincides with the center of mass; this is equivalent to the mutual distances being given by $r_{ij}^2=\hbox{constant}\bigl(\frac{1}{m_i}+\frac{1}{m_j}\bigr)$.
\smallskip

\noindent{\bf Remark.} 
More generally, the proof of proposition \ref{DegIn} implies that higher degeneracies with $\nu_2$ pairs, $\nu_3$ triples, etc$\ldots$, of equal eigenvalues, instead of having the generic codimension $\sum{\frac{1}{2}(i-1)(i+2)\nu_i}=2\nu_2+5\nu_3+9\nu_4+\cdots$ have codimension $\sum{(i-1)\nu_i}=\nu_2+2\nu_3+3\nu_4+\cdots$
\goodbreak

\section{The case of four bodies}\label{Four}
The unique non planar central configuration of four arbitrary masses is the regular tetrahedron.
We show that, whatever be the masses, it belongs to at least three (and exactly three in the generic case)
1-parameter (up to rotation and scaling) families of balanced configurations which admit a relative equilibrium motion in $\R^4$. Two cases have to be treated separately, depending of whether or not three of the masses are equal (see corollary \ref{exact3} and section \ref{3=}).  

We use the following notations for the squared mutual distances:
$$r_{13}^2=a, \; r_{14}^2=b', \; r_{12}^2=b'', \; r_{34}^2=d', \; r_{32}^2=d'', \; r_{24}^2=f.$$
They were chosen because they specialize nicely to the case when the masses $m_2$ and $m_4$ are symmetric with respect to the plane containing $m_1$ and $m_3$ (see section \ref{Sym} and \cite{C3}). 
\begin{center}
\includegraphics[scale=1]{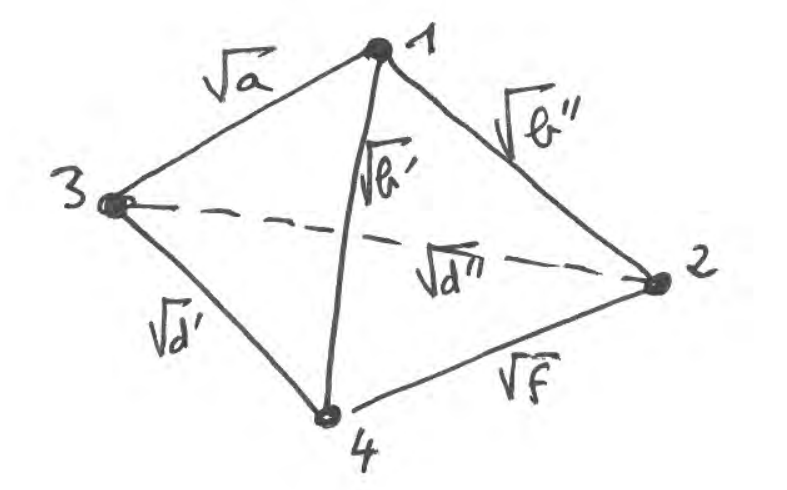} 

Figure 1:  4-body configurations.
\end{center}
\goodbreak

\noindent A convenient $\mu^{-1}$-orthonormal basis of ${\mathcal D}^*$ is $\{u_1,u_2,u_3\}$ with
\begin{equation*}
\left\{
\begin{split}
u_1&=\kappa_1\bigl(m_1(m_2+m_4), -m_2(m_1+m_3), m_3(m_2+m_4), -m_4(m_1+m_3)\bigr),\\
u_2&=\kappa_2(1,0,-1,0),\\
u_3&=\kappa_3(0,1,0,-1),
\end{split}
\right.
\end{equation*}
where
$$
\kappa_1^2=\frac{1}{M(m_1+m_3)(m_2+m_4)},\quad \kappa_2^2=\frac{m_1m_3}{m_1+m_3},\quad 
\kappa_3^2=\frac{m_2m_4}{m_2+m_4}. 
$$

\subsection{The intrinsic inertia matrix}\label{SubIn}
Expressed in the basis $\{u_1,u_2,u_3\}$, 
$B=\begin{pmatrix}
u&z&y\\ z&v&x\\ y&x&w
\end{pmatrix}, $
where
\begin{equation*}
\left\{
\begin{split}
u&=\frac{1}{M}\bigl[m_2(m_1b''+m_3d'')+m_4(m_1b'+m_3d')\bigr]-\frac{1}{M}\left[m_1m_3\frac{Y}{Z}a+m_2m_4\frac{Z}{Y}f\right],\\
v&=\frac{m_1m_3}{m_1+m_3}a,\quad w=\frac{m_2m_4}{m_2+m_4}f,\\
x&=\frac{1}{2}V(d''-d'+b'-b''),\\
y&=\frac{U}{2MZ}\bigl[m_1(b'-b'')+m_3(d'-d'')\bigr]+\frac{U(m_4-m_2)}{2MY}f,\\
z&=\frac{T}{2MY}\bigl[m_2(b''-d'')+m_4(b'-d')\bigr]+\frac{T(m_1-m_3)}{2MZ}a,
\end{split}
\right.
\end{equation*}
with the following notations ($X=MYZ=TU/V$ is for future use in \ref{WCgeneral}): 
\begin{equation*}
\left\{
\begin{split}
X&=\frac{M}{(m_1+m_3)(m_2+m_4)},\; &&Y=\frac{1}{m_1+m_3},\\
Z&=\frac{1}{m_2+m_4},\; &&T=\frac{1}{m_1+m_3}\sqrt{\frac{Mm_1m_3}{m_2+m_4}},\\
U&=\frac{1}{m_2+m_4}\sqrt{\frac{Mm_2m_4}{m_1+m_3}},\; &&V=\sqrt{\frac{m_1m_2m_3m_4}{(m_1+m_3)(m_2+m_4)}}\, \cdot
\end{split}
\right.
\end{equation*}
For the regular tetrahedron $(a=b'=b''=d'=d''=f=1)$,  $B$ becomes the matrix
$$\begin{pmatrix}
\frac{(m_2+m_4)(m_1^2+m_3^2)}{2M(m_1+m_3)}+\frac{(m_1+m_3)(m_2^2+m_4^2)}{2M(m_2+m_4)}&
-\frac{1}{2}\sqrt{\frac{m_1m_3(m_2+m_4)}{M}}\frac{m_3-m_1}{m_1+m_3}&
-\frac{1}{2}\sqrt{\frac{m_2m_4(m_1+m_3)}{M}}\frac{m_2-m_4}{m_2+m_4}\\
-\frac{1}{2}\sqrt{\frac{m_1m_3(m_2+m_4)}{M}}\frac{m_3-m_1}{m_1+m_3}&
\frac{m_1m_3}{m_1+m_3}&0\\
-\frac{1}{2}\sqrt{\frac{m_2m_4(m_1+m_3)}{M}}\frac{m_2-m_4}{m_2+m_4}&0&\frac{m_2m_4}{m_2+m_4}
\end{pmatrix}$$
which reduces to $B=\frac{1}{2}Id$ when all the masses are equal to 1. 

\subsection{Degeneracies of the inertia of the regular tetrahedron}\label{EquDeg}
As explained in section \ref{deg}, we shall use the criterion given by lemma \ref{Lax}, which leads to much more transparent equations than the resultant:
$B$ is degenerate if and only if there exists a non trivial antisymmetric matrix
$$
R=\begin{pmatrix}
0&\zeta&-\eta\\
-\zeta&0&\xi\\
\eta&-\xi&0
\end{pmatrix},
$$
which commutes with $B$, that is such that $[B,R]=BR-RB=0$. Writing the six coefficients of this symmetric matrix in the order $(u,v,w,x,y,z)$, this is equivalent to the following linear equation:
$$
\begin{pmatrix}
0&2y&-2z\\
-2x&0&2z\\
2x&-2y&0\\
v-w&-z&y\\
z&w-u&-x\\
-y&x&u-v
\end{pmatrix}
\begin{pmatrix}
\xi\\ \eta\\ \zeta
\end{pmatrix}
=
\begin{pmatrix}
0\\ 0\\ 0\\ 0\\ 0\\ 0
\end{pmatrix}.
$$
This equation has a non trivial solution if and only if the above $6\times 3$ matrix is not of maximal rank, that is if its rows generate a subspace of dimension 2 or less. We shall distinguish two cases:
\smallskip

1) two of the off-diagonal coefficients $x,y,z$ of $B$ are equal to zero. Up to a reordering of the basis, we may suppose that $x=y=0$. Then the condition is reduced to the vanishing of two minors: 
$(u-v)\bigl[z^2+(v-w)(w-u)\bigr]$ and $z\bigl[z^2+(v-w)(w-u)\bigr]$, that is 
\begin{equation*}
(C_1)\left\{
\begin{split}
\hbox{either}&\; z=0\quad \hbox{and}\; \quad (u-v)(v-w)(w-u)=0,\\
\hbox{or}&\; z\not=0\quad \hbox{and}\; \quad z^2+(v-w)(w-u)=0.
\end{split}
\right.
\end{equation*}
\smallskip

2) at most one of the off-diagonal coefficients is equal to zero. Then the first two lines are linearly independent while the sum of the first three is equal to zero (because the trace of a commutator vanishes). Writing that the last three lines belong to the plane generated by the first two leads to the following equations
\begin{equation*}
(C_2)\left\{
\begin{split}
x(y^2-z^2)+(v-w)yz=0,\quad\quad(C_2')\\
y(z^2-x^2)+(w-u)zx=0,\quad\quad(C_2'')\\
z(x^2-y^2)+(u-v)xy=0.\quad\quad(C_2''')
\end{split}
\right.
\end{equation*}
As expected, these equations are not independent: multiplying the first by $x$, the second by $y$, the third by $z$ and adding, one gets 0. Recall that they are not valid if two off-diagonal coefficients vanish.
\begin{proposition}\label{DegIn}
The intrinsic inertia matrix of the regular tetrahedron configuration of four masses $m_1,m_2,m_3,,m_4$ is degenerate if and only if at least three of the masses are equal.
\end{proposition}

\begin{proof} The if part follows without calculation from the fact that an ellipsoid with $\Z/3\Z$-symmetry is necessary of revolution. For the converse, as the off-diagonal coefficient $x$ is equal to $0$,  equations $(C_2)$ imply that another off-diagonal coefficient must be equal to zero, that is $m_1=m_3$ or $m_2=m_4$. Hence equations $(C_2)$ are no more pertinent to decide of the degeneracy. If we suppose $m_2=m_4$, they must be replaced by equations $(C_1)$:

1) if all off-diagonal coefficients are zero, that is if $m_2=m_4$ and $m_1=m_3$, we have
$$B=\hbox{diag}\left(\frac{m_1m_2}{m_1+m_2}, \frac{m_1}{2},\frac{m_2}{2}\right),$$ 
hence degeneracy occurs only if the four masses are equal;

2) if $m_1\not=m_3$, a direct computation shows that $(C_1)$ reads
$$
-\frac{m_2}{2M}(m_1-m_2)(m_3-m_2)=0,
$$
which implies that $m_1=m_2=m_4$ or $m_2=m_3=m_4$.
Starting with $m_1=m_3$ but $m_2\not=m_4$ we would have found the remaining possibilities $m_1=m_2=m_3$ and $m_1=m_3=m_4$.
\end{proof}

\noindent If, for example, $m_2=m_3=m_4$, we have
$$
B=\begin{pmatrix}
\frac{m_2(3m_1^2+3m_2^2+2m_1m_2)}{2(m_1+3m_2)(m_1+m_2)}&-\frac{m_2}{2}\sqrt{\frac{2m_1}{m_1+3m_2}}\frac{m_2-m_1}{m_1+m_2}&0\\
-\frac{m_2}{2}\sqrt{\frac{2m_1}{m_1+3m_2}}\frac{m_2-m_1}{m_1+m_2}&\frac{m_1m_2}{m_1+m_2}&0\\
0&0&\frac{m_2}{2}
\end{pmatrix}.
$$
whose eigenvalues are 
$$\lambda_1=\frac{2m_1m_2}{m_1+3m_2},\quad \lambda_2=\lambda_3=\frac{m_2}{2}\; .$$

\begin{corollary}\label{4bodyGenType}
The unique 3-dimensional central configuration of 4 bodies -- the regular tetrahedron -- is of general type if and only if no three of the masses are equal.
\end{corollary}
\begin{proof} 
The regular tetrahedron is the unique critical point of the restriction of the potential to the 4-body configurations with fixed moment of inertia with respect to the centre of mass (i.e. fixed trace of $B$). Using the squared mutual distances $r_{ij}^2$ as independent variables, it is easy to prove that it is a non-degenerate minimum. Hence its restriction to the isospectral manifold is also a non degenerate minimum.
\end{proof}
\smallskip

\noindent From proposition \ref{PropGenType} on then deduces
\begin{corollary}\label{exact3}
If no three masses are equal, the balanced 4-body configurations (up to rotation and scaling) close to the regular tetrahedron form a 2-dimensional manifold; the balanced configurations which admit a relative equilibrium motion in $\R^4$ form three regular curves intersecting at the regular tetrahedron.
\end{corollary}
\noindent In what follows we give a direct proof of corollary \ref{exact3}, computing in particular the curves of degenerate balanced configurations, first in case 2 masses are equal (section \ref{Sym}), then, at first order, in the general case (section \ref{Trip}).

\subsection{The Wintner-Conley matrix}\label{WCgeneral}  
A tedious but straightforward computation gives the following expression of the Wintner-Conley endomorphism in this basis:

$$
A=\begin{pmatrix}
\alpha&\phi&\epsilon\\
\phi&\beta&\delta\\
\epsilon&\delta&\gamma
\end{pmatrix},\quad\quad\hbox{where (with the notations of \ref{SubIn})}$$

\begin{equation*}
\left\{
\begin{split}
\alpha&=X\left[m_2\bigl(m_1\varphi(b'')+m_3\varphi(d'')\bigr)+
m_4\bigl(m_1\varphi(b')+m_3\varphi(d')\bigr)\right],\\
\beta&=(m_1+m_3)\varphi(a)+Y\left[m_2\bigl(m_3\varphi(b'')+m_1\varphi(d'')\bigr)+
m_4\bigl(m_3\varphi(b')+m_1\varphi(d')\bigr)\right],\\
\gamma&=(m_2+m_4)\varphi(f)+Z\left[m_1\bigl(m_2\varphi(b')+m_4\varphi(b'')\bigr)+
m_3\bigl(m_2\varphi(d')+m_4\varphi(d'')\bigr)\right],\\
\delta&=V\left[\varphi(d'')-\varphi(d')+\varphi(b')-\varphi(b'')\right],\\
\epsilon&=U\left[m_1\bigl(\varphi(b')-\varphi(b'')\bigr)
+m_3\bigl(\varphi(d')-\varphi(d'')\bigr)\right],\\
\phi&=T\left[m_2\bigl(\varphi(b'')-\varphi(d'')\bigr)
+m_4\bigl(\varphi(b')-\varphi(d')\bigr)\right].
\end{split}
\right.
\end{equation*}
\noindent For the regular tetrahedron with unit sides, we find $A=M\varphi(1)Id=-\frac{M}{2}Id$.
\goodbreak
 
\subsection{The equations of balanced configurations}\label{CE}
With the above notations, the 4 equations of balanced configurations take the following form, where one checks that, in accordance with
section \ref{Balcon}, they satisfy the relation
$P_{123}-P_{124}+P_{134}-P_{234}\equiv 0.$
\goodbreak

\begin{equation*}
(P_{123})\left\{
\begin{split}
&m_1(d''-a-b'')[\varphi(a)-\varphi(b'')]-m_4(d''+b')[\varphi(f)-\varphi(d')]\\
+&m_2(a-b''-d'')[\varphi(b'')-\varphi(d'')]-m_4(a+f)[\varphi(d')-\varphi(b')]\\
+&m_3(b''-d''-a)[\varphi(d'')-\varphi(a)]-m_4(b''+d')[\varphi(b')-\varphi(f)]=0,
\end{split}
\right.
\end{equation*}
\begin{equation*}
(P_{124})\left\{
\begin{split}
&m_1(f-b'-b'')[\varphi(b')-\varphi(b'')]-m_3(f+a)[\varphi(d'')-\varphi(d')]\\
+&m_2(b'-b''-f)[\varphi(b'')-\varphi(f)]-m_3(b'+d'')[\varphi(d')-\varphi(a)]\\
+&m_4(b''-f-b')[\varphi(f)-\varphi(b')]-m_3(b''+d')[\varphi(a)-\varphi(d'')]=0,
\end{split}
\right.
\end{equation*}
\begin{equation*}
(P_{134})\left\{
\begin{split}
&m_1(d'-b'-a)[\varphi(b')-\varphi(a)]-m_2(d'+b'')[\varphi(d'')-\varphi(f)]\\
+&m_3(b'-a-d')[\varphi(a)-\varphi(d')]-m_2(b'+d'')[\varphi(f)-\varphi(b'')]\\
+&m_4(a-d'-b')[\varphi(d')-\varphi(b')]-m_2(a+f)[\varphi(b'')-\varphi(d'')]=0,
\end{split}
\right.
\end{equation*}
\begin{equation*}
(P_{234})\left\{
\begin{split}
&m_2(d'-f-d'')[\varphi(f)-\varphi(d'')]-m_1(d'+b'')[\varphi(a)-\varphi(b')]\\
+&m_3(f-d''-d')[\varphi(d'')-\varphi(d')]-m_1(f+a)[\varphi(b')-\varphi(b'')]\\
+&m_4(d''-d'-f)[\varphi(d')-\varphi(f)]-m_1(d''+b')[\varphi(b'')-\varphi(a)]=0.
\end{split}
\right.
\end{equation*}

\noindent Let us use these equations to give a direct proof of the first part of corollary \ref{exact3}~: linearized at the regular tetrahedron whose sides have length 1, the equations of balanced configurations take a particularly simple form, {\it independent of the precise form of} $\varphi$:
$$
\varphi'(1)K\begin{pmatrix}
\delta a\\ \delta b'\\ \delta b''\\ \delta d'\\ \delta d''\\ \delta f
\end{pmatrix}
=
\begin{pmatrix}
0\\ 0\\ 0\\ 0
\end{pmatrix},\quad\hbox{where}
$$
$$
K=\begin{pmatrix}
m_3-m_1&0&m_1-m_2&0&m_2-m_3&0\\
0&m_4-m_1&m_1-m_2&0&0&m_2-m_4\\
m_1-m_3&m_4-m_1&0&m_3-m_4&0&0\\
0&0&0&m_3-m_4&m_2-m_3&m_4-m_2
\end{pmatrix}.
$$
The rank of the matrix $K$ is 0 if the four masses are equal, 2 if three of the masses are equal and 3 otherwise. In this last case, its kernel is generated by the three vectors
\begin{equation*}
\left\{
\begin{split}
E_1&=(1,1,1,1,1,1), \\
E_2&=(m_2m_4,m_2m_3,m_3m_4,m_1m_2,m_1m_4,m_1m_3),\\
E_3&=(m_2+m_4,m_2+m_3,m_3+m_4,m_1+m_2,m_1+m_4,m_1+m_3).
\end{split}
\right.
\end{equation*}
As the rank at neighboring points cannot be higher than three because the equations we have used of the set of balanced configurations are not independent, the rank of the matrix is locally constant if no three of the masses are equal.
\goodbreak

\smallskip

\subsection{The $\Z/2\Z$-symmetric case}\label{Sym}
In this section, we suppose that at least two masses are equal, say $m_2=m_4$, which makes pertinent the use of the basis $\{u_1,u_2,u_3\}$ introduced at the beginning of section \ref{Four}. We check directly that in this case the degeneracy of the Wintner-Conley matrix becomes a codimension 1 condition (Corollary \ref{exact3}) among balanced configurations close to the regular tetrahedron. We start with a nice corollary of the fact that the intrinsic inertia matrices of these balanced configurations form a 3-dimensional manifold: 
\begin{corollary}\label{DynSym}
If two masses and not three are equal, any balanced configuration close enough to the regular tetrahedron is symmetric with respect to a plane separating these two masses and containing the other two. 
\end{corollary}
To distinguish from a purely geometric symmetry, we shall call {\it dynamically symmetric} such a symmetric configuration for which symmetric masses are equal.
\vskip0.1cm

\begin{proof}Supposing $m_2=m_4$, let us consider the naturally associated symmetric configurations, which satisfy  (figure 2)
$$b'=b''=b,\quad d'=d''=d.$$
\begin{center}
\includegraphics[scale=0.8]{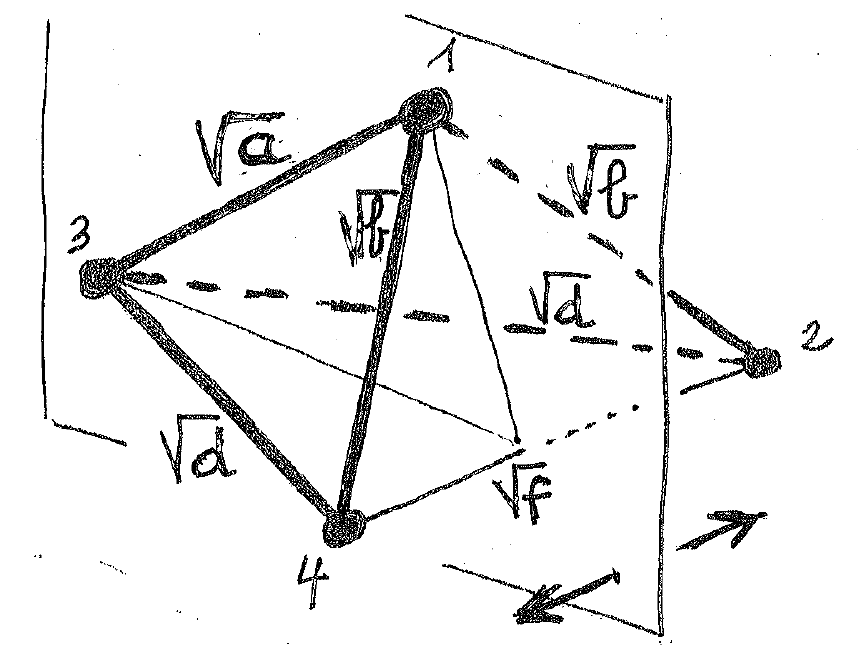}

Figure 2:  symmetric 4-body configurations.
\end{center}
\goodbreak

\noindent The four equations $P_{ijk}$ reduce to a single one in the 4 variables $(a,b,d,f)$. Indeed, 
$P_{234}$ and $P_{124}$ are identically satisfied, while $P_{123}=-P_{134}=0$ becomes
\begin{equation*}
\left\{
\begin{split}
&m_1(d-b-a)[\varphi(b)-\varphi(a)]-m_2(d+b)[\varphi(d)-\varphi(f)]\\
+&m_3(b-a-d)[\varphi(a)-\varphi(d)]-m_2(b+d)[\varphi(f)-\varphi(b)]\\
+&m_2(a-d-b)[\varphi(d)-\varphi(b)]-m_2(a+f)[\varphi(b)-\varphi(d)]=0.
\end{split}
\right.
\end{equation*}
Linearized at the regular tetrahedron, the equation becomes
$$
(m_1-m_3)\delta a+(m_2-m_1)\delta b+(m_3-m_2)\delta d=0.
$$
Hence the equation is a submersion at the regular tetrahedron, except in case all masses are equal. It follows that, under our hypotheses, the set of symmetric balanced configurations  in the neighborhood of the regular tetrahedron  is  a submanifold
whose dimension is the same as the one of all balanced configurations. Hence the two coincide locally.
\end{proof}
\smallskip

\noindent With the same proof, based on a dimension count, we get
\begin{corollary}\label{losanges} If the masses form two equal pairs, any balanced configuration close enough to the regular tetrahedron is a {\rm rhombus configuration}, i.e. it is symmetric with respect to two orthogonal planes respectively separating two equal masses and containing the other two. Supposing  $m_2=m_4\not=m_1=m_3$, this means that $b'=b''=d'=d''$.  
\end{corollary}

\noindent Corollary \ref{DynSym} makes it easy to get a good understanding of the degenerate balanced configurations near the regular tetrahedron when two masses and not three are equal. Indeed, suppose as above that $m_2=m_4,\; b'=b''=b,\;d'=d''=d$. 
In the $\mu^{-1}$-orthonormal basis $\{u_1,u_2,u_3\}$ of $\mathcal D^*$ defined at the beginning of section \ref{Four}, 
{\it the Wintner-Conley matrix decomposes into two blocks:}

\begin{equation*}
A=\begin{pmatrix}
\alpha&\phi&0\\
\phi&\beta&0\\
0&0&\gamma
\end{pmatrix},\quad\hbox{with}\quad
\left\{
\begin{split}
\alpha&=\frac{M}{m_1+m_3}\bigl(m_1\varphi(b)+m_3\varphi(d)\bigr),\\
\beta&=\frac{2m_2}{m_1+m_3}\bigl(m_3\varphi(b)+m_1\varphi(d)\bigr)+(m_1+m_3)\varphi(a),\\
\gamma&=m_1\varphi(b)+m_3\varphi(d)+2m_2\varphi(f),\\
\phi&=\frac{\sqrt{2m_1m_2m_3M}}{m_1+m_3}\bigl(\varphi(b)-\varphi(d)\bigr).
\end{split}
\right.
\end{equation*}
Degeneracy occurs  if  (compare to \ref{EquDeg}, condition $(C_1)$) 
either the upper block degenerates, i.e.  
$$\phi=0,\; \alpha=\beta,$$ 
or the lower right coefficient is equal to an eigenvalue of the upper block, i.e. 
$$\phi^2+\left(\frac{\alpha-\beta}{2}\right)^2-\left(\gamma-\frac{\alpha+\beta}{2}\right)^2=0.$$

\noindent {\it (i)} The first case is equivalent to  $a=b=d$ and the equation of balanced configurations is automatically satisfied.
\smallskip

\noindent {\it (ii)} In order to study the second case, we introduce the following coordinates:
$$\alpha,\quad \theta=\gamma-\frac{\alpha+\beta}{2},\quad \psi=\frac{\alpha-\beta}{2},\quad \phi,$$
that is
$$\beta=\alpha-2\psi,\quad \gamma=\alpha+\theta-\psi.$$
In these coordinates, the degenerate matrices other than the ones defined by $\psi=\phi=0$ are defined by the equation
$$\phi^2+\psi^2-\theta^2=0.$$
On the other hand, the Wintner-Conley matrices of balanced configurations are defined by the equation $P_{124}=P_{234}$, where $a,b,c,d$ are expressed in terms of the new coordinates $\alpha,\theta,\psi,\phi$ via the inverse equations (where for saving space we have noted $s_{13}=m_1+m_3$; thanks to Jacques F\'ejoz for the computation of the inverse matrix),

\begin{equation*}
\left\{\begin{split}
\varphi(a)&=\bigl(\frac{1}{M}-\frac{1}{s_{13}}\bigr)\alpha+\frac{1}{s_{13}}\beta-\frac{m_3-m_1}{s_{13}^2T}\phi 
&&=\frac{1}{M}\alpha-\frac{2}{s_{13}}\psi-\frac{m_3-m_1}{s_{13}^2T}\phi\,,\\
\varphi(b)&=\frac{1}{M}\alpha+\frac{m_3}{2m_2s_{13}T}\phi, 
&&=\frac{1}{M}\alpha+\frac{m_3}{2m_2s_{13}T}\phi\,,\\
\varphi(d)&=\frac{1}{M}\alpha-\frac{m_1}{2m_2s_{13}T}\phi 
&&=\frac{1}{M}\alpha-\frac{m_1}{2m_2s_{13}T}\phi\,, \\
\varphi(f)&=\bigl(\frac{1}{M}-\frac{1}{2m_2}\bigr)\alpha+\frac{1}{2m_2}\gamma
&&=\frac{1}{M}\alpha+\frac{1}{2m_2}\theta-\frac{1}{2m_2}\psi.\\
\end{split}
\right.
\end{equation*}

\noindent We have seen that, if no three masses are equal, the balanced configurations near the regular tetrahedron ($a=b=d=f=l$) form a 3-dimensional manifold whose tangent space at this point is defined by the equation
$$(m_1-m_3)\delta a+(m_2-m_1)\delta b+(m_3-m_2)\delta d=0.$$
On the side of the Wintner-Conley matrices, that is setting
\begin{equation*}
\left\{\begin{split}
\varphi'(l)\delta a&=\frac{1}{M}\delta\alpha-\frac{2}{s_{13}}\delta\psi-\frac{m_3-m_1}{s_{13}^2T}\delta\phi\,,\\
\varphi'(l)\delta b&=\frac{1}{M}\delta\alpha+\frac{m_3}{2m_2s_{13}T}\delta\phi\,,\\
\varphi'(l)\delta d&=\frac{1}{M}\delta\alpha-\frac{m_1}{2m_2s_{13}T}\delta\phi\,, \\
\varphi'(l)\delta f&=\frac{1}{M}\delta\alpha+\frac{1}{2m_2}\delta\theta-\frac{1}{2m_2}\delta\psi,\\
\end{split}
\right.
\end{equation*}
this equation becomes
$$
-\frac{2(m_1-m_3)}{m_1+m_3}\delta\psi+\frac{3m_2(m_1^2+m_3^2)-2m_1m_3(m_1+m_2+m_3)}{2m_2(m_1+m_3)^2T}\delta\phi=0.
$$
\goodbreak 

\noindent It defines a linear susbspace ``transversal" to the cone defined by the degenerate matrices. Indeed, 
It is enough to look in the space of coordinates $(\theta,\psi,\phi)$ obtained by going to the quotient by the $\alpha$ axis : in the coordinates $(\alpha,\beta,\gamma,\phi)$, this corresponds to going to the quotient by the line generated by $(1,1,1,0)$, that is by the addition to the Wintner-Conley matrix of a multiple of the identity, which does not change its degeneracy type and leaves invariant the linearized equations of balanced configurations  (see  section \ref{Trip}). 
Hence we have identified the three (projective) directions of bifurcation to relative equilibria in $\R^4$ garanteed by corollary \ref{exact3}\smallskip

\begin{center}
\includegraphics[scale=0.5]{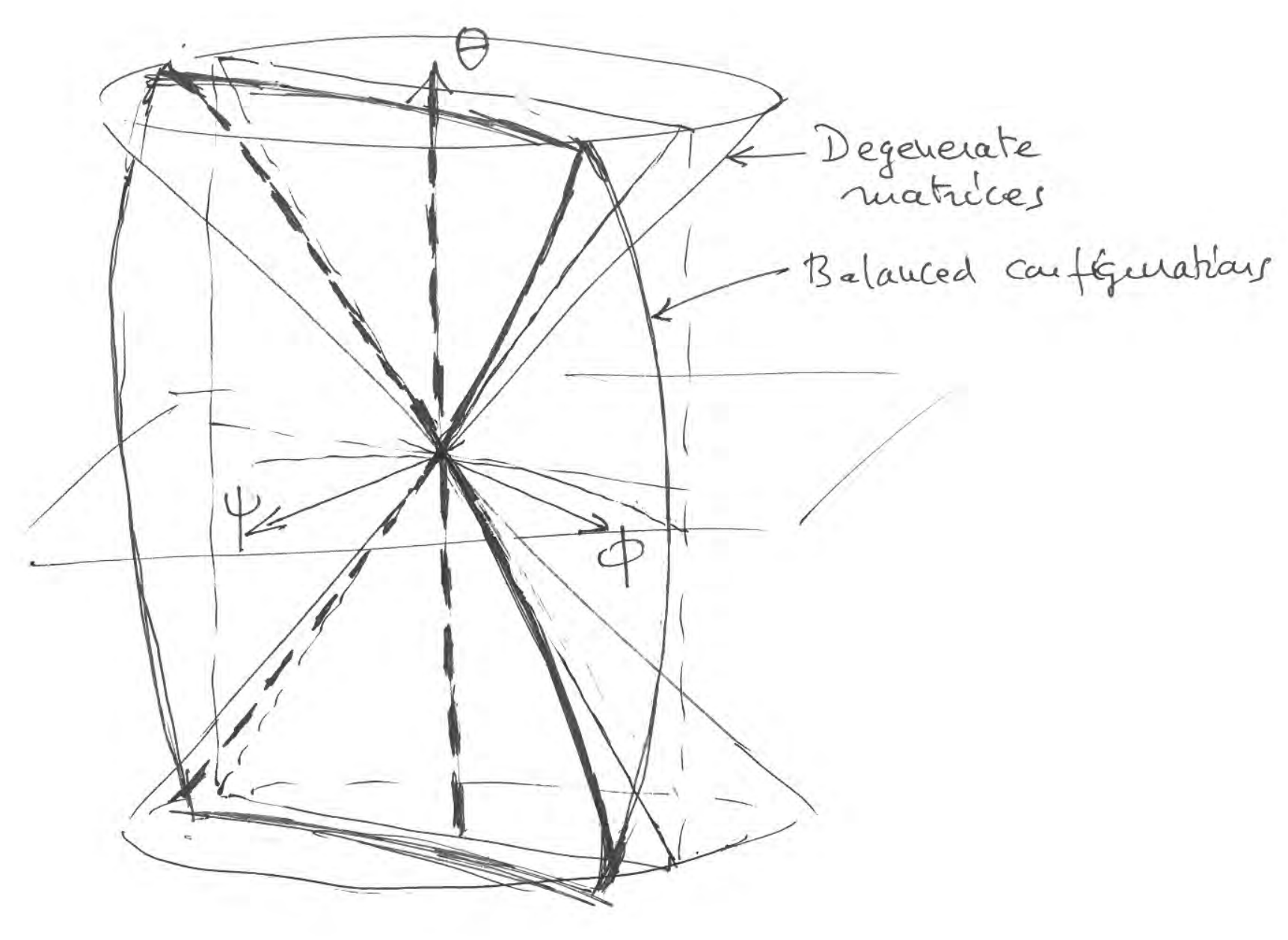}

Figure 3:  Degenerate balanced configurations in the symmetric case.
\end{center}

\noindent {\bf Remark.} The degenerate rhombus configurations are explicitely given by one of the following three equations:
$$a=b,\quad \hbox{or}\quad f=b,\quad\hbox{or}\quad (m_1-m_2)\varphi(b)=m_1\varphi(a)-m_2\varphi(f).$$

\subsection{The case of three equal masses}\label{3=}
\noindent Suppose for example that $m_2=m_3=m_4$; we have the following explicit (projective) curves of 
degenerate symmetric balanced configurations:
\smallskip

{\it 1) The ($\Z/2\Z$-symmetric) configurations:}
we have seen in section \ref{Sym} that, supposing only $m_2=m_4$, the configurations such that $a=b'=b''=d'=d''$
are balanced and degenerate. In the basis $\{u_1,u_2,u_3\}$, their Wintner Conley matrix is
$$diag\bigl(M\varphi(a),M\varphi(a),(m_1+m_3)\varphi(a)+2m_2\varphi(f)\bigr).$$
From the two other equalities $m_2=m_3$ (resp. $m_3=m_4$), one gets two other families of $\Z/2\Z$-symmetric degenerate balanced configurations, namely the configurations such that $a=b'=b''=d'=f$ (resp. $a=b'=b''=d''=f$).

{\it 2) The ($\Z/3\Z$-symmetric) configurations:} if $a=b'=b'',\; d'=d''=f$
 and assuming only $m_2=m_4$, the unique equation of balanced configurations reduces to $a(m_2-m_3)\bigl(\varphi(a)-\varphi(f)\bigr)=0$.
Supposing moreover $m_2=m_3$, it is natural to choose the following $\mu^{-1}$-orthonormal basis of $\mathcal D^*$:

\begin{equation*}
\begin{split}
v_1&=\sqrt{\frac{m_1m_2}{3(m_1+3m_2)}}(-3,1,1,1),\;  v_2=\sqrt{\frac{m_2}{6}}(0,-1,2,-1),\\
v_3&=\frac{m_2}{2}(0,1,0,-1).
\end{split}
\end{equation*}
In this base, the Wintner-Conley matrix is
$$diag\bigl((m_1+3m_2)\varphi(a),m_1\varphi(a)+3m_2\varphi(f),m_1\varphi(a)+3m_2\varphi(f)\bigr).$$
\smallskip

\noindent {\bf Remarks.} 
1) When three and not four masses are equal, non symmetric planar balanced configurations do exist: the idea, communicated to me by Alain Albouy is to start with the central configuration consisting in an isosceles triangle with the fourth mass a little above the center of mass and break the symmetry by slightly changing one mass in the basis of the isosceles triangle. It is not known whether non-symmetric balanced configurations exist in the neighborhood of the regular tetrahedron with three equal masses. 

\smallskip

 2) The case of four equal masses is studied in \cite{C3}. The set of symmetric balanced configurations is singular at the regular tetrahedron: there are four symmetry planes; fixing one of them, the subset of balanced configurations with this symmetry plane is the union of two (projective) surfaces with transversal intersection: one of these consists in the rhombus configurations and contains three (projective) curves of degenerate configurations; the other one contains the curve of  $\Z/3\Z$-symmetric degenerate configurations. It is not known whether non-symmetric 4-body balanced configurations exist in the case of equal masses.

\subsection{The general case: TRIP comes into play}\label{Trip}

We have seen in section \ref{CE} that, as soon as $\varphi'(1)\not=0$, the linearization at the regular tetrahedron of the equations of balanced configurations takes the form
$$
K\begin{pmatrix}
\delta a\\ \delta b'\\ \delta b''\\ \delta d'\\ \delta d''\\ \delta f
\end{pmatrix}
=
\begin{pmatrix}
0\\ 0\\ 0\\ 0
\end{pmatrix},\quad\hbox{where $K$ is defined in \ref{CE}.}
$$
We suppose that the masses are all different, hence the rank of $K$ equals 3 and its kernel is generated by the three vectors $E_1,E_2,E_3$ defined in section \ref{CE}.
\goodbreak

\noindent Computing a Taylor expansion of $A$ in the neighborhood of the regular tetrahedron with unit sides, one gets
$$
A=-\frac{M}{2}
\begin{pmatrix}
1&0&0\\ 0&1&0\\ 0&0&1
\end{pmatrix}
+
\begin{pmatrix}
\delta\alpha&\delta\phi&\delta\epsilon\\
\delta\phi&\delta\beta&\delta\delta\\
\delta\epsilon&\delta\delta&\delta\gamma
\end{pmatrix},\quad\hbox{with}\quad \begin{pmatrix}\delta\alpha\\ \delta\beta\\ \delta\gamma\\ \delta\delta\\ \delta\epsilon\\ \delta\phi
\end{pmatrix}=\varphi'(1)L\begin{pmatrix}
\delta a,\\ \delta b'\\ \delta b''\\ \delta d'\\ \delta d''\\ \delta f
\end{pmatrix},
$$
$$
\hbox{and}\quad L=\begin{pmatrix}
0&Xm_1m_4&Xm_1m_2&Xm_3m_4&Xm_2m_3&0\\
m_1+m_3&Ym_3m_4&Ym_2m_3&Ym_1m_4&Ym_1m_2&0\\
0&Zm_1m_2&Zm_1m_4&Zm_2m_3&Zm_3m_4&m_2+m_4\\
0&V&-V&-V&V&0\\ 
0&Um_1&-Um_1&Um_3&-Um_3&0\\
0&Tm_4&Tm_2&-Tm_4&-Tm_2&0
\end{pmatrix},
$$
\noindent where we have used the notations of \ref{WCgeneral}. 
Hence, the tangent space at the regular tetrahedron of the manifold of Wintner-Conley matrices of balanced configurations
is generated by the three vectors in $\R^6$ (coordinates $(\alpha,\beta,\gamma,\delta,\epsilon,\phi)$)

\begin{equation*}
\begin{split}
\hskip-2.5cm 
L(E_1)=M\begin{pmatrix}
1\\ 1\\ 1\\ 0\\ 0\\ 0
\end{pmatrix},\; 
L(E_2)=
\begin{pmatrix}
4XM\\
m_2m_4\bigl[m_1+m_3+2Y(m_1^2+m_3^2)\bigr]\\
m_1m_3\bigl[m_2+m_4+2Z(m_2^2+m_4^2)\bigr]\\
V(m_2-m_4)(m_3-m_1)\\
2Um_1m_3(m_2-m_4)\\
2Tm_2m_4(m_3-m_1)
\end{pmatrix},\\
L(E_3)=
\begin{pmatrix}
2X\sum_{i<j<k}m_im_jm_k\\
(m_1+m_3)(m_2+m_4)+\frac{m_2+m_4}{m_1+m_3}(m_1^2+m_3^2)+2m_2m_4\\
(m_2+m_4)(m_1+m_3)+\frac{m_1+m_3}{m_2+m_4}(m_2^2+m_4^2)+2m_1m_3\\
0\\
U(m_1+m_3)(m_2-m_4)\\
T(m_2+m_4)(m_3-m_1)
\end{pmatrix}.
\end{split}
\end{equation*}

\noindent{\bf Remark.} The fact that $E_1$ belongs to the kernel of the linearized equations $K$ of balanced configurations can be seen directly from the equation $[A,B]=0$. Indeed, the linearized equation at a central configuration, $[A,\Delta B]+[\Delta A,B]=0$, is satisfied because $A$ is a multiple of Identity as is $\Delta A=L(E_1)$. 
\smallskip

\noindent As the equations of degenerate quadratic forms are also invariant by the addition of a multiple of  the identity matrix, that is  a multiple of $L(E_1)$ in the $(\alpha,\beta,\gamma,\delta,\epsilon,\phi)$ space, in order to understand the tangent spaces at the regular tetrahedron to the degenerate balanced configurations, it is enough to substitute $L(E_2)+xL(E_3)$ 
to $(\alpha,\beta,\gamma,\delta,\epsilon,\phi)$ in the three equations $(C_2)$ of section \ref{EquDeg} (after replacing $u,v,w,x,y,z$ respectively by $\alpha,\beta,\gamma,\delta,\epsilon,\phi$). and this is what I had asked Jacques Laskar to do. 
\smallskip

\noindent Using the computer algebra software TRIP developped by him and Micka\"el Gastineau\cite{GL}, he discovered that the three equations became proportional, namely he obtained
\begin{equation*}
\left\{
\begin{split}
&\frac{(m_1-m_3)(m_2-m_4)}{(m_1+m_3)(m_2+m_4)}\sqrt{\frac{m_1m_2m_3m_4}{(m_1+m_3)(m_2+m_4)}}\times  \cdots\\
&\quad\quad\quad\cdots 2M(-m_3m_2m_4-m_1m_2m_4+m_1m_2m_3+m_1m_3m_4)E=0,\\
&-\frac{(m_1-m_3)^2(m_2-m_4)}{m_1+m_3}m_1m_3\sqrt{\frac{Mm_2m_4}{m_1+m_3}}E=0,\\
&-\frac{(m_1-m_3)(m_2-m_4)^2}{m_2+m_4}m_2m_4\sqrt{\frac{Mm_1m_3}{m_2+m_4}}E=0,
\end{split}
\right.
\end{equation*}
where
$$E(x):=(\sum_i m_i)x^3+2(\sum_{i<j}m_im_j)x^2+3(\sum_{i<j<k}m_im_jm_k)x+4\prod_i m_i=0.$$
The equation $E$ has three real roots, all negative. Indeed, if we set
\begin{equation*}
\begin{split}
F(y)&=(1+m_1y)(1+m_2y)(1+m_3y)(1+m_4y)\\
&= (\prod_i m_i)y^4+(\sum_{i<j<k}m_im_jm_k)y^3+(\sum_{i<j}m_im_j)y^2+(\sum_i m_i)y+1,
\end{split}
\end{equation*}
we have
$$F'(y)= 4(\prod_i m_i)y^3+3(\sum_{i<j<k}m_im_jm_k)y^2+2(\sum_{i<j}m_im_j)y +(\sum_i m_i),$$
and hence
$$E(x)=x^3F'(1/x).$$
One concludes because $F$ has the 4 real roots $-1/m_1,-1/m_2,-1/m_3,-1/m_4$, all negative.
This computation gives the tangents to the three (projective) curves of degenerate balanced configurations which intersect at the  regular tetrahedron.  As I already said in the abstract, this came as a surprise as I had asked Jacques to show that, for masses without any symmetry, no other solution than the regular tetrahedron existed locally, in accordance with the generic crossing of eigenvalues of symmetric matrices being of codimension two. This surprise was the incentive to prove proposition \ref{PropGenType} and corollary \ref{exact3}. 
\medskip

\noindent {\bf Remark.} The equation $E(x)=0$ remains pertinent in the symmetric case studied in section \ref{Sym}, where $m_2=m_4,\; b'=b'',\; d'=d''$,\; (hence  $\delta=\epsilon=0$), and no three masses are equal :  the polynome $F$ has a double root $-1/m_2=-1/m_4$, hence one of the roots of $E$ is $-m_2=-m_4$. The corresponding sides
$$E_2-m_2E_3=(-m_2^2,-m_2^2,-m_2^2,-m_2^2,-m_2^2,m_1m_3-m_1m_2-m_2m_3),$$
are such that $a=b=d$, hence 
their Wintner-Conley matrix is diagonal with $\alpha=\beta$. This was the first case in the study of \ref{Sym}, and the corresponding direction satisfies the equations of balanced configurations (and not only at the first order). 
\smallskip

\noindent In the even more special case of rhombus relative equilibria, that is when $$m_2=m_4\not= m_1=m_3, \quad b'=b''=d'=d'',$$
the three roots 
$$-m_1,\quad -m_2,\quad -\frac{2m_1m_2}{m_1+m_2},$$
of the equation $E(x)=0$ give back the three directions of degeneracy of the Wintner-Conley matrix; indeed:
\begin{equation*}
\left\{
\begin{split}
&\hbox{if}\; x=-m_1,\quad &&a=m_2^2-2m_1m_2,\quad b=f=-m_1^2,\\
&\hbox{if}\; x=-m_2,\quad &&a=b=-m_2^2,\quad f=m_1^2-2m_1m_2,\\
&\hbox{if}\; x=-\frac{2m_1m_2}{m_1+m_2},\quad &&a=\frac{m_2^3-3m_1m_2^2}{m_1+m_2},\; b=-m_1m_2,\; f=\frac{m_1^3-3m_1^2m_2}{m_1+m_2}.
\end{split}
\right.
\end{equation*}
The first two cases correspond to actual lines of degeneracy of the Wintner-Conley matrix, while in the third case, $(m_1-m_2)b=m_1a-m_2f$ gives only the tangent to the actual degeneracy curve $(m_1-m_2)\varphi(b)=m_1\varphi(a)-m_2\varphi(f)$. 
\vskip0.2cm

\noindent In case 3 masses are equal, say $m_2=m_3=m_4:=m,$ one has $E_3=mE_1+\frac{1}{m}E_2$ and, indeed, $E_2$ is the line of degenerate $\Z/3\Z$-symmetric balanced configurations described in section \ref{3=}.
\vskip0.2cm

\noindent Finally, when the four masses are equal, the three vectors $E_1,E_2,E_3$ are proportional and hence they yield only trivial information.

\section{The angular momentum}
To a point $(x,y)$ in the phase space (more accurately the tangent space to the configuration space) of the $n$-body problem, is attached its {\it angular momentum bivector} (\cite{AC,C1,C2})
$$c=-x\circ\mu\circ y^{tr} + y\circ\mu\circ x^{tr} \in\wedge^2E. $$
Transformed to the endomorphism ${\mathcal C}=c\circ\epsilon$ of $E$, it is represented in an orthonormal basis by the antisymmetric matrix
$${\mathcal C}=-XY^{tr}+YX^{tr}.$$
If $Y=\Omega X$ as this is the case for a relative equilibrium motion $X(t)=e^{t\Omega}X(0)$, we get
$${\mathcal C}=-S\Omega^{tr}+\Omega S=S\Omega+\Omega S,$$
where the inertia matrix $S$ of the configuration is defined in \ref{Balcon}.

\subsection{The frequency polytope and its subpolytopes}
{\it We now fix a central configuration $x_0$. In this case, the relative equilibria are all periodic, of the form $x(t)=e^{\tilde\omega Jt}x_0$, where $J$ is a complex structure on $E\equiv\R^{2p}$ (see \ref{dimMin}). By scaling the configuration, one may even assume that the frequency $\tilde\omega$ is equal to 1.} As soon as the dimension of the ambient space is 4 or more, the same central configuration admits a whole family of relative equilibrium motions parametrized by the $J$'s. In case the intrinsic inertia $B$ of the configuration is generic, these motions can be characterized by their angular momentum.
\smallskip

\noindent To the central configuration $x_0=\{\vec r_1,\cdots,\vec r_n\}$ is naturally attached (see\cite{C1,CJ,HZ}) a convex polytope ${\cal P}$ contained in the $(p-1)$ simplex $$\left\{(\nu_1,\cdots,\nu_p)\in(\R_+)^p,\nu_1\ge\cdots\ge\nu_p, \sum_{i=1}^p\nu_i=
\frac{1}{\sum_{i=1}^pm_i}\sum_{1\le i<j\le p}m_im_jr_{ij}^2\right\},$$
where we recall that the $r_{ij}=||\vec r_i-\vec r_j||$ are the mutual distances between the bodies.
This polytope is the set of ordered $p$-tuples $(\nu_1\ge\cdots\ge\nu_p)$ of positive real numbers such that 
$\{\pm i\nu_1,\cdots,\pm i\nu_p\}$ is the spectrum of the $J$-skew-hermitian matrix $SJ+JS$ representing the angular momentum of the relative equilibrium motion of $x_0$ defined by some $J$.  
Once chosen an orthonormal basis of $E\equiv\R^{2p}$, $\cal P$ can be described as the image of the  {\it frequency map}
$${\cal F}:U(p)/SO(2p)\to W_p^+\{(\nu_1,\cdots,\nu_p),\nu_1\ge\cdots\ge\nu_p\}$$
which, to a complex structure $J\in U(p)/SO(2p)$, that is to an identification of $E$ with $\C^p$ such that the multiplication by $i$ is an isometry, associates the ordered spectrum of the $J$-hermitian matrix $J^{-1}S_0J+S_0$, where $S_0$ 
is the inertia matrix of the chosen configuration. Up to some zeroes coming from the difference in dimensions, the inertia $S_0=X_0X_0^{tr}$ and the intrinsic inertia $B_0=X_0^{tr}X_0$ have the same spectrum. Recall in particular that
their common trace is the moment of inertia with respect to the center of mass; by a formula of Leibniz, it is equal to 
$$I_0=trace\; S_0=\frac{1}{\sum_{i=1}^pm_i}\sum_{1\le i<j\le p}m_im_jr_{ij}^2.$$
As it depends only on the inertia matrix $S_0$ of the configuration, the polytope ${\cal P}$ is  defined as well for any solid body,
\smallskip

\noindent It was proved in \cite{C1,CJ} that ${\cal P}$ is a {\it Horn polytope}, more precisely that if $\{\sigma_1\ge\cdots\ge\sigma_{2p}\}$ is the ordered spectrum of $S_0$, ${\cal P}$ is the set of ordered spectra of real symmetric $p\times p$ matrices of the form $c=a+b$, where
$$\hbox{spectrum}\,(a)=\{\sigma_1,\sigma_3,\cdots,\sigma_{2p-1}\},\quad \hbox{spectrum}\,(b)=\{\sigma_2,\sigma_4,\cdots,\sigma_{2p}\}.$$  
Choosing other partitions $\Pi$ of the spectrum of $S_0$ into two subsets with $p$ elements, one defines in the same way Horn polytopes ${\cal P}_\Pi$ which turn out to be subpolytopes of ${\cal P}$ (this is a non trivial fact\footnote{Intuitively, each piece of the partition defining ${\mathcal P}$ is as ``separated" as possible; in contrast, in corollary \ref{geve} the polytope associated to the partition  $\{\nu_1,\nu_2,\cdots,\nu_d\}\sqcup\{0,0,\cdots,0\}$ is reduced to a point.} which is proved in \cite{FFLP}). It was noticed in \cite{C1} that, given some (periodic) relative equilibrium of a central configuration, a bifurcation to a family of (quasi-periodic) relative equibria of balanced configurations can occur only if the corresponding point in the frequency polytope ${\cal P}$ lies in some face of one of these subpolytopes  ${\cal P}_\Pi$. We are interested in identifying the faces which actually correspond to such bifurcations.

\subsection{The generic bifurcation vertex}
From proposition \ref{extrinsic} we deduce 
\begin{corollary}\label{geve}
In the situation of Proposition \ref{extrinsic}, the ordered frequencies of the angular momentum of $x(t)$ are $(\nu_1=\sigma_1\ge\nu_2=\sigma_2\ge\cdots\ge\nu_d=\sigma_d)$. They correspond to a vertex on the boundary of the frequency polytope \cite{C1,CJ}.
\end{corollary}

\noindent The identification of the frequencies $\nu_i=\sigma_i+0$ of the angular momentum is an immediate consequence of the the nature of $J$ in the proposition.
It remains to prove that the corresponding vertex is a boundary vertex and not an interior one. It is enough to prove that it is a vertex of the polytope $\mathcal{P}$ associated to the partition $\sigma_-\cup\sigma_+$ of the spectrum of the inertia matrix $S_0$ with  
$\sigma_-=\{\sigma_1,\sigma_3,\cdots,0\cdots,0\}$ and $\sigma_+=\{\sigma_2,\sigma_4,\cdots,0,\cdots,0\}$, where 
$\sigma_1\ge\cdots \sigma_d$ (see \cite{C1,CJ}). 
This comes from the fact that whatever be $d$, odd or even, the number of non zero terms in $\sigma_-$ (resp. $\sigma_+$) is the same as the number of zeroes in $\sigma_+$  (resp. $\sigma_-$); hence there is a vertex of $\mathcal{P}$ which corresponds to a permutation coupling each $\sigma_i$ with a 0.
\medskip

\noindent {\bf Exemple: Three bodies.}
\smallskip

\noindent In the equal mass three-body problem, the bifurcation from an equilateral periodic relative equilibrium of a family of isosceles quasi-periodic relative equilibria in $\R^4$ with  2 frequencies cannot originate from the planar Lagrange solution but only from an equilateral relative equilibrium whose angular momentum is equivalent to the complex structure~$J_0$. 
\begin{center}
\includegraphics[scale=0.6]{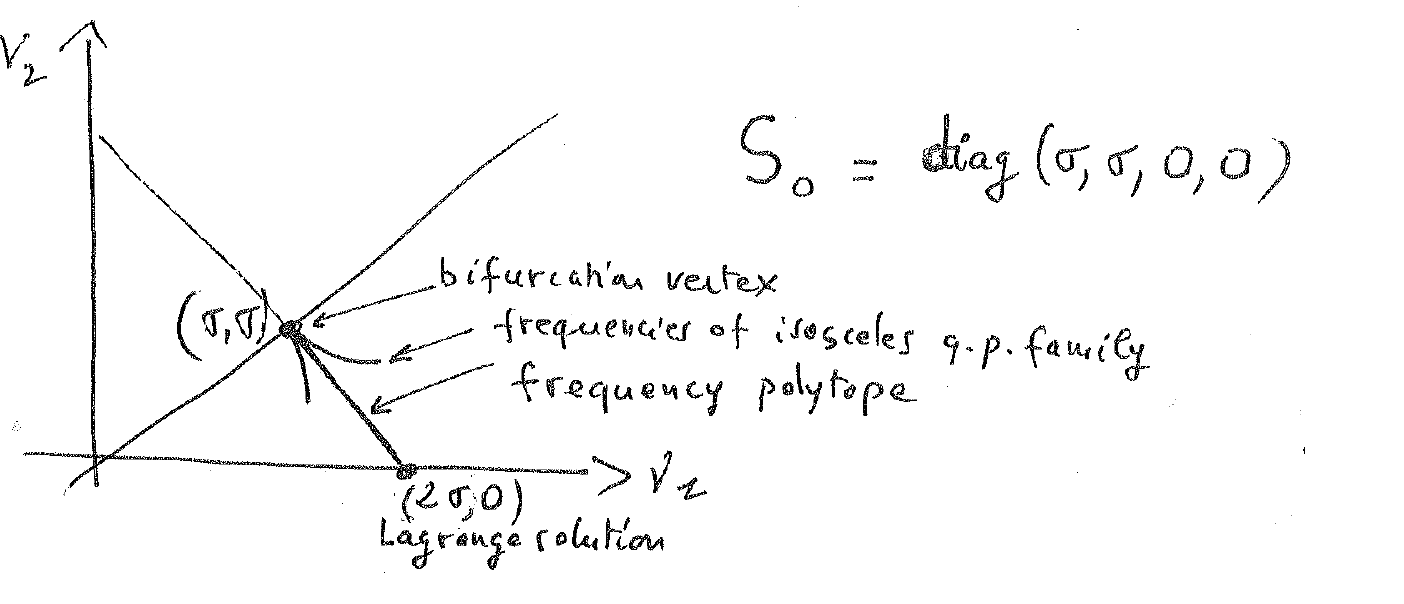}

Figure 4: Bifurcation from periodic equilateral to quasi-periodic isosceles
\end{center}

\subsection{Bifurcation locus in the frequency polytope} 
\noindent Figure 5 depicts the frequency polytope of the regular tetrahedron configuration in $\R^6$. Generically, only two possibilities exist for the inertia eigenvalues $\sigma_1, \sigma_2, \sigma_3$:
$$
(1)\;\, \sigma_1>\sigma_2+\sigma_3>\sigma_2>\sigma_3>0\quad\hbox{or}\quad (2)\;\, \sigma_2+\sigma_3>\sigma_1>\sigma_2>\sigma_3>0.
$$
The first case is what becomes the example depicted in \cite{C1} under the assumption that $\sigma_4=\sigma_5=\sigma_6=0$. An example is the regular tetrahedron with one of the masses much smaller than the three others. Another one is the regular tetrahedron with masses $m_1=m_3>>m_2=m_4$. An example of the second one is the regular tetrahedron with almost equal masses.
\smallskip

\noindent Figure 6 indicates the angular momentum frequencies corresponding to the sizes and vertices when the frequency of the corresponding relative equilibrium equals 1 (if not, all the frequencies should be multiplied by this frequency $\omega$).
\smallskip

\noindent There are 4 distinct partitions $\Pi$ of the spectrum $\{\sigma_1,\sigma_2,\sigma_3,0,0,0\}$ :

1) $\Pi_0=\{\sigma_1,\sigma_2,\sigma_3\}\cup\{0,0,0\}$; the corresponding Horn polytope is reduced to one point, 
the ``generic bifurcation vertex" $A$, which is the only place where a periodic relative equilibrium in $\R^6$ of the given central configuration could bifurcate into a family of quasi-periodic relative equilibria in $\R^6$ with 3 frequencies of balanced configurations with 
$\lambda_1,\lambda_2,\lambda_3$ all distinct\footnote{The notations $A,B,C,D$ for the vertices have, of course, no relation with the matrices $A, B,C$.}. 
\smallskip

2) $\Pi_i=\{\sigma_i,0,0\}\cup\{\sigma_j,\sigma_k,0\},\, i=1,2,3,$ the frequency polytope, which contains all the others, corresponding to $i=2$.  

\noindent The vertex $B$ corresponds to a periodic relative equilibrium motion in $\R^4$ which could bifurcate into a family of quasi-periodic relative equilibria with 2 frequencies  in $\R^4$ of balanced configurations with  $\lambda_1=\lambda_2\not=\lambda_3$.  The eigenplanes of the instantaneous rotation $\Omega$ would tend respectively to 
$\{\rho_1,\rho_2\}$ and $\{\rho_3, v_1\}$, where in agreement with the notations in \ref{dimMin}, $v_1$ is any non-zero vector orthogonal to the image of the balanced configuration in question. 
It follows that the bifurcation happens at the vertex $B$, which corresponds to a relative equilibrium directed by the complex structure having the planes $\{\rho_1,\rho_2\}$ and $\{\rho_3,v_1\}$ as complex lines.

\begin{center}
\includegraphics[scale=0.6]{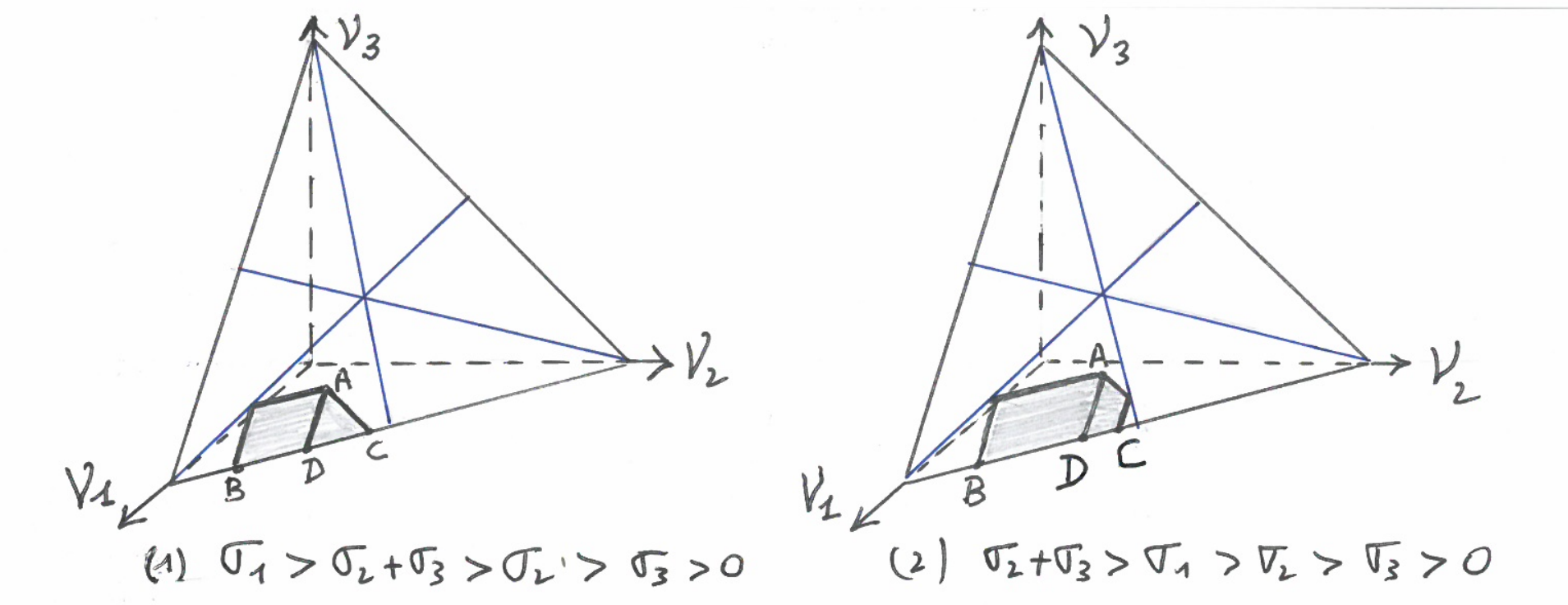}

Figure 5: Bifurcation loci in the generic cases
\end{center}

\begin{center}
\includegraphics[scale=0.6]{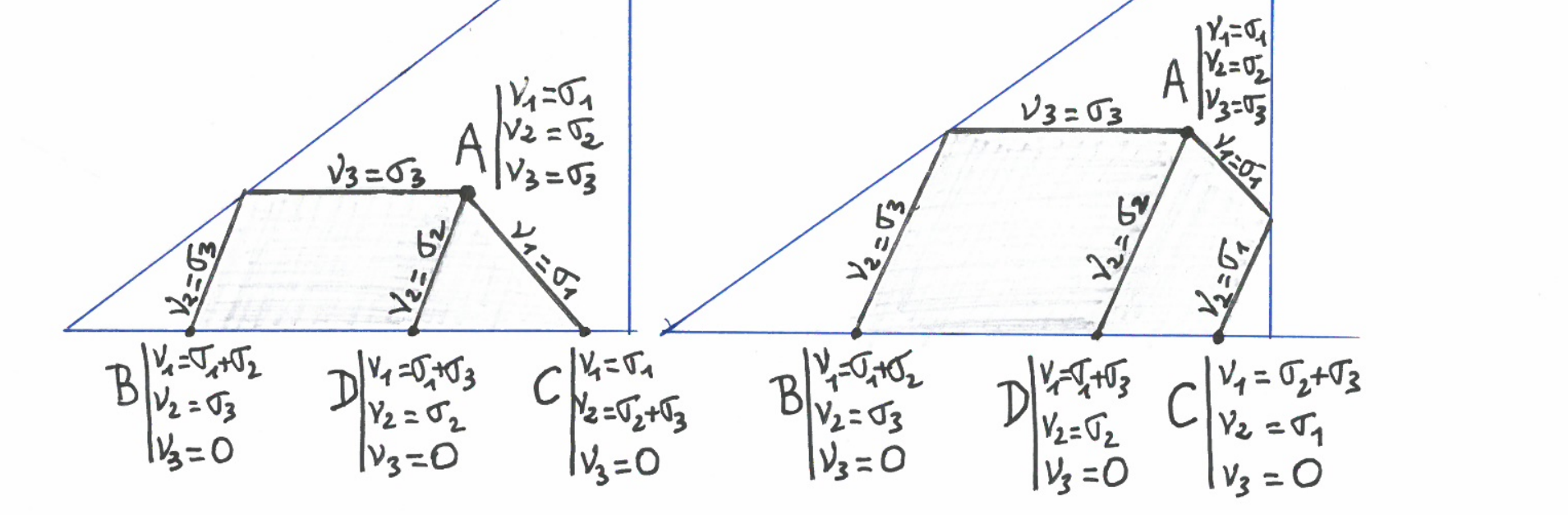}

Figure 6: Angular momentum frequencies (when rotation frequency = 1)
\end{center}

\noindent Analogous descriptions hold for the vertex $C$ ($\lambda_1\not=\lambda_2=\lambda_3$) and the vertex $D$ ($\lambda_1=\lambda_3\not=\lambda_2$). 
\smallskip

\noindent The broken edge $AB$ ($\nu_3=\sigma_3$ reflected in $\nu_2=\sigma_3$) corresponds to periodic relative equilibria in $\R^6$ which could bifurcate into a family of quasi-periodic relative equilibria with 2 frequencies  in $\R^6$ of balanced configurations with  $\lambda_1=\lambda_2\not=\lambda_3$ and whose instantaneous rotation $\Omega$ would have a 4-dimensional eigenspace $\{\rho_1,\rho_2,v_1,v_2\}$ and a 2-dimensional eigenplane $\{\rho_3,v_3\}$. The edge is parametrized by the choice of a complex structure in the 4-dimensional eigenspace.
\smallskip

\noindent In the same way, the (possibly broken) edge $AC$ ($\nu_1=\sigma_1$) corresponds to periodic relative equilibria in $\R^6$ which could bifurcate into a family of quasi-periodic relative equilibria with 2 frequencies  in $\R^6$ of balanced configurations with  $\lambda_1\not=\lambda_2=\lambda_3$ and whose instantaneous rotation $\Omega$ would have a 4-dimensional eigenspace $\{\rho_2,\rho_3,v_2,v_3\}$ and a 2-dimensional eigenplane $\{\rho_1,v_1\}$. The edge is, as above, parametrized by the choice of a complex structure in the 4-dimensional eigenspace.
Finally, the same description holds for the interior side $AD$. 
\smallskip

\noindent On the contrary, apart from the vertices $B,D,C$, the edge $BC$ ($\nu_3=0$) does not correspond to possible bifurcations. This is because it is the interior of the frequency polytope when the dimension of $E$ goes down to 4. This remark indicates in more general situations what faces of the frequency polytope (and subpolytopes) are bifurcation values.
\goodbreak

\section{Rhombus 4-body relative equilibria}
\subsection{The 3d case}
According to corollary \ref{losanges}, any balanced configuration close enough to the regular tetrahedron with only 2 different masses, say $m_1=m_3\not=m_2=m_4$,  is a rhombus configuration :
$$r_{12}=r_{14}=r_{32}=r_{34}=\sqrt{b},\quad r_{13}=\sqrt{a}\quad ,r_{24}=\sqrt{f}.$$
\smallskip

\noindent In such a simple case, it is possible to give explicit formul{\ae} for the bifurcating families.
Given real numbers $(\alpha,\beta,\gamma_1,\gamma_2)$, which we may suppose all positive,  such that $m_1\gamma_1=m_2\gamma_2$, we define a configuration $x_0$ of 4 bodies in $\R^3$ with center of mass at the origin by
$$x_0=
\begin{pmatrix}
\alpha&0&-\alpha&0\\ 
0&\beta&0&-\beta\\ 
\gamma_1&-\gamma_2&\gamma_1&-\gamma_2\\ 
\end{pmatrix}.
$$

The mutual distances are
\begin{equation*} 
\left\{
\begin{split}
\sqrt{a}&=r_{13}=2\alpha,\; \sqrt{f}= r_{24}=2\beta,\\
\sqrt{b}&=\sqrt{d}=r_{12}=r_{14}=r_{32}=r_{34}=(\alpha^2+\beta^2+(\gamma_1+\gamma_2)^2)^{1/2},
\end{split}
\right.
\end{equation*} 
In the $\mu^{-1}$-orthonormal basis $\{u_1,u_2,u_3\}$ of $\mathcal{D}^*$ formed by the vectors
\begin{equation*}
u_1=\sqrt{\frac{m_1m_2}{2(m_1+m_2)}}(1,-1,1-1),\; 
u_2=\sqrt{\frac{m_1}{2}}(1,0,-1,0),
u_3=\sqrt{\frac{m_2}{2}}(0,1,0,-1),
\end{equation*}
$x_0$ is represented by the $3\times 3$ matrix whose columns are Jacobi vectors
$$X_0=
\begin{pmatrix}
0&\sqrt{2m_1}\alpha&0\\ 
0&0&\sqrt{2m_2}\beta\\ 
\sqrt{\frac{2m_1m_2}{m_1+m_2}}(\gamma_1+\gamma_2)&0&0 

\end{pmatrix}.
$$ 
The corresponding  inertia matrices are 
\begin{equation*}
B_0=X_0^{tr}X_0=
\begin{pmatrix}
\sigma_3&0&0\\ 
0&\sigma_1&0\\ 
0&0&\sigma_2
\end{pmatrix},\quad 
S_0=X_0X_0^{tr}=
\begin{pmatrix}
\sigma_1&0&0\\
0&\sigma_2&0\\
0&0&\sigma_3\\
\end{pmatrix},\end{equation*}
with
\begin{equation*}
\left\{\begin{split}
\sigma_1&=2m_1\alpha^2=\frac{m_1}{2}a,\; \sigma_2=2m_2\beta^2=\frac{m_2}{2}f,\\
\sigma_3&=
\frac{2m_1m_2}{m_1+m_2}(\gamma_1+\gamma_2)^2=2(m_1\gamma_1^2+m_2\gamma_2^2)=\frac{m_1m_2}{2(m_1+m_2)}(4b-a-f),
\end{split}
\right.
\end{equation*}
while the Wintner-Conley endomorphism $A:\mathcal{D}^*\to\mathcal{D}^*$ is
$$A=
\begin{pmatrix}
(m_1+m_2)\varphi(b)&0&0\\
0&\bigl(m_1\varphi(a)+m_2\varphi(b)\bigr)&0\\
0&0&\bigl(m_1\varphi(b)+m_2\varphi(f)\bigr)
\end{pmatrix}.$$

\subsubsection{Bifurcations in $\R^6$}
{\it i) Bifurcations from the generic vertex.} We embed the configuration $x$ in $\R^6$ by equaling to 0 the last 3 coordinates of each body and
identify $(q_1,\dots,q_6)\in \R^6$ with $(z_1=q_1+iq_4,\;  z_2=q_2+iq_5,\;  z_3=q_3+iq_6)\in \C^3$.
A relative equilibrium is defined by making each column of $X$ move according to 
$$(Z_1,Z_2,Z_3)\mapsto (e^{i\omega_1t}Z_1,\; e^{i\omega_2t}Z_2, e^{i\omega_3t}Z_3),$$
with 
$$\omega_1^2=-2(m_1+m_2)\varphi(b),\omega_2^2=-2\bigl(m_1\varphi(a)+m_2\varphi(b)\bigr),
\omega_3^2=-2\bigl(m_1\varphi(b)+m_2\varphi(f)\bigr),$$
where
$(Z_1,Z_2,Z_3)\in C^3$ are the columns of $X$, considered as belonging to $\mathcal{D}\otimes\R^6$ (that is with three zeros added). 
Its angular momentum $\mathcal{C}$ is (compare to \cite{C1})
\begin{equation*}
\hskip-1.5cm
\begin{pmatrix}
0&0&0&-\sigma_1\omega_1&0&0\\
0&0&0&0&-\sigma_2\omega_2&0\\
0&0&0&0&0&-\sigma_3\omega_3\\
\sigma_1\omega_1&0&0&0&0&0\\
0&\sigma_2\omega_2&0&0&0&0\\
0&0&\sigma_3\omega_3&0&0&0
\end{pmatrix}.
\end{equation*}

\noindent This gives a 2-parameter\footnote{$a,b,f$ modulo scaling} family bifurcating at the ``generic vertex" 
$$A=\bigl(\frac{m_1}{2},\frac{m_1m_2}{m_1+m_2},\frac{m_2}{2}\bigr)$$ from the regular tetrahedron with unit sides and masses $m_1,m_2,m_1,m_2$ such that $m_1+m_2=1$ (and hence $\omega^2=1$) and $m_1>m_2$.  
\medskip

\noindent {\it 2) Bifurcations from the sides of the frequency polytope.} The relative equilibria bifurcating as above from the generic vertex have 3 frequencies except when one of the equalities $a=b$ or $b=f$ of $(m_1-m_2)\varphi(b)=m_1\varphi(a)-m_2\varphi(f)$ holds, in which case, only two distinct frequencies survive.
The missing frequency is in some sense replaced by the parameter along one side of the frequency polytope (or subpolytope), which corresponds to the latitude of choice (in fact a 2-sphere) of the complex structure which directs the relative equilibrium from which the family bifurcates. Namely, supposing $a=b=1$ and $m_1+m_2=1$, and embedding the configuration $x$ in $\R^6$ via the embedding $(q_1,q_2,q_3)\mapsto (q_1,0,q_2,0,q_3,0)$ of $\R^3$ in $\R^6$, we define a relative equilibrium by setting
$$x(t)=\begin{pmatrix}
\cos\omega t&0&-\cos\theta\sin\omega t&-\sin\theta\sin\omega t&0&0\\
0&\cos\omega t&\sin\theta\sin\omega t&-\cos\theta\sin\omega t&0&0\\
\cos\theta\sin\omega t&-\sin\theta\sin\omega t&\cos\omega t&0&0&0\\
\sin\theta\sin\omega t&\cos\theta\sin\omega t&0&\cos\omega t&0&0\\
0&0&0&0&\cos\omega_3 t&-\sin\omega_3 t&\\
0&0&0&0&\sin\omega_3 t&\cos\omega_3 t
\end{pmatrix}x_0,
$$
where $\omega^2=-2(m_1+m_2)\varphi(b)=1,\, \omega_3^2=-2\bigl(m_1\varphi(b)+m_2\varphi(f)\bigr)=m_1+m_2f^{-\frac{3}{2}}$. Then
$$S_0=\begin{pmatrix}
\sigma_1&0&0&0&0&0\\
0&0&0&0&0&0\\
0&0&\sigma_2&0&0&0\\
0&0&0&0&0&0\\
0&0&0&0&\sigma_3&0\\
0&0&0&0&0&0
\end{pmatrix}$$
and 
$$C=\begin{pmatrix}
0&0&-(\sigma_1+\sigma_2)\omega\cos\theta&-\sigma_1\omega\sin\theta&0&0\\
0&0&\sigma_2\omega\sin\theta&0&0&0\\
(\sigma_1+\sigma_2)\omega\cos\theta&-\sigma_2\omega\sin\theta&0&0&0&0\\
\sigma_1\omega\sin\theta&0&0&0&0&0\\
0&0&0&0&0&-\sigma_3\omega_3\\
0&0&0&0&\sigma_3\omega_3&0
\end{pmatrix},$$
whose frequencies are $\sigma_3\omega_3$ and the square roots of the solutions of $$\nu^2+[\sigma_1^2+\sigma_2^2+2\sigma_1\sigma_2\cos^2\theta]\omega^2\nu+\sigma_1^2\sigma_2^2\omega^4\sin^4\theta=0.$$
At the bifurcation, when $\theta$ varies from $\pi/2$ to $0$, they vary from $(\sigma_1,\sigma_2,\sigma_3)$ to $(\sigma_1+\sigma_2,0,\sigma_3)$, which makes $(\nu_1,\nu_2,\nu_3)$ travel the broken side $AB$ in figure 6.
\smallskip

\noindent{\bf Remark.} In agreement with the results of \cite{C1,CJ}, it is an ``adapted" family of complex structures which has been chosen to direct the relative equilibria of the regular tetrahedron at the bifurcation: they send the plane $\{\rho_1,v_1\}$ onto the plane $\{\rho_2,v_2\}$ and $\rho_3$ onto $v_3$. 

\subsubsection{Bifurcations in $\R^4$}

\noindent {\it i) The case a=b.} This is the limit $\theta=0$ of the above family: the complex structure sends $\rho_1$ to $\rho_2$, $\rho_3$ to $v_3$ (and $v_1$ to $v_2$, which implies that the motion does not visit the corresponding dimensions). 
After embedding the configuration in $\R^4$ by equaling to 0 the last coordinate of each body and identifying
$(q_1,q_2,q_3,q_4)\in\R^4$ with 
$(z_1=q_1+iq_2, z_2=q_3+iq_4)\in\C^2$, such a relative equilibrium motion is defined by each column of $X$ moving according to 
$(Z_1,Z_2)\mapsto(e^{i\omega_1 t}Z_1, e^{i\omega_3 t}Z_2),$
with $\omega_1^2=-2(m_1+m_2)\varphi(b),\; \omega_3^2=-2(m_1\varphi(b)+m_2\varphi(f))$.
Its angular momentum $ \mathcal{C}$ is
$$\begin{pmatrix}
0&-(\sigma_1+\sigma_2)\omega_1&0&0\\
(\sigma_1+\sigma_2)\omega_1&0&0&0\\
0&0&0&-\sigma_3\omega_3\\
0&0&\sigma_3\omega_3&0\end{pmatrix}.$$
\smallskip

\noindent {\it ii) The case $(m_1-m_2)\varphi(b)=m_1\varphi(a)-m_2\varphi(f)$ and the case $b=f$.} 
The situation is analogous, the only difference being the identification of $\R^4$ with $\C^2$ which is respectively
$z_1=q_2+iq_3, z_2=q_1+iq_4$ and $z_1=q_1+iq_3, z_2=q_2+iq_4$. The angular momentum spectra  are respectively
$\bigl\{(\sigma_2+\sigma_3)\omega_1,\; \sigma_1\omega_2\bigr\}$ and $\bigl\{(\sigma_1+\sigma_3)\omega_1,\; \sigma_2\omega_2\bigr\}.$

\subsection{The 2d case}\label{2d} This is the case when $\gamma_1=\gamma_2=0$, that is $4b-a-f=0$. We check property {\bf (H)} (see \ref{general}) for the symmetric ($m_1=m_3,\; b'=b''=b, \;  d'=d''=d$, see \ref{Sym}) balanced configurations in the neighborhood of the planar rhombus central configuration $x_0$. As $Im B_0$ is the plane $x=0$, $x_0$ is characterized by the equality of the last two eigenvalues of $A_0$, that is
$$m_1\varphi(a)-m_2\varphi(f)=(m_1-m_2)\varphi(b)=(m_1-m_2)\varphi(\frac{a+f}{4}).$$

\begin{lemma} Supposing $m_1=m_3$ and $m_2=m_4$, let $x_0$ be a planar rhombus (hence balanced) configuration  and let $K$ be defined by 
$$K=2m_1\bigl(\varphi(a_0)-\varphi(b_0)\bigr)-(m_1+m_2)a_0\varphi'(b_0).$$
If $K\not=0$, the set of planar $\Z{/2\Z}$-symmetric 4-body balanced configurations close to $x_0$ coincide with the set of planar rhombus configurations and the condition {\bf (H)} of section \ref{general} is satisfied at $x_0$ for these planar $\Z/2\Z$-symmetric balanced configurations. 
\end{lemma}
\begin{proof}
One linearizes at $(a,b,d,f)=(a_0,b_0,b_0,f_0)$ the couple formed by the unique equation of $\Z/2\Z$-symmetric balanced configurations (see \ref{Sym}) and the Cayley-Menger determinant, proportional to the squared volume of the configuration; the first assertion of the lemma follows because the $2\times 4$ matrix one gets is of the form
 $$det\begin{pmatrix}
 0&K&-K&0\\
 ...&...&...&-2a_0f_0
 \end{pmatrix}.$$
Hence, if $K\not=0$, the set inertia matrices $B$ of balanced configurations close to $x_0$ is a 2 dimensional submanifold parametrized by $a$ and $f$ (or by the non-zero eigenvalues $\frac{1}{2}m_1a$ and $\frac{1}{2}m_2f$ of $B$). 
The mapping sending the non zero eigenvalues of the inertia matrix of a planar balanced configuration $B$ to the spectrum of $A|_{Im B}$ then reduces to
$$(a,f)\mapsto \left(2m_1\varphi(a)+2m_2\varphi\bigl(\frac{a+f}{4}\bigr), 2m_2\varphi(f)+2m_1\varphi\bigl(\frac{a+f}{4}\bigr)\right),$$
whose derivative at $B_0$ is always invertible because its determinant 
$$\bigl(m_1^2\varphi'(a_0)+m_2^2\varphi'(f)\bigr)\varphi'\bigl(\frac{a_0+f_0}{4}\bigr)+2m_1m_2\varphi'(a_0)\varphi'(f_0)$$
is strictly positive. Hence condition {\bf (H)} is satisfied (note that at this point we have not to suppose that $x_0$ is central).   
\end{proof}
\smallskip

\noindent Finally, the ellipsoid of inertia $B_0$ of a rhombus planar configuration is degenerate (i.e. round) if and only if $m_1a=m_2f$. Taking the Newtonian value $\varphi(s)=-\frac{1}{2}s^{-\frac{3}{2}}$ and supposing that $x_0$ is central, the condition of degeneracy becomes
$$m_1-m_1^{-\frac{3}{2}}m_2^{\frac{5}{2}}+8m_2^{\frac{3}{2}}(m_2-m_1)(m_1+m_2)^{-\frac{3}{2}}=0,$$ or, normalizing the masses by setting $m_2=1$,
$$m_1-m_1^{-\frac{3}{2}}+8(1-m_1)(1+m_1)^{-\frac{3}{2}}=0,$$
which defines 3 values $\gamma, 1, 1/\gamma$, with $\gamma\simeq 0.575$. Hence, 
\begin{lemma}
Except when $m_1/m_2$ equals $\gamma, 1/\gamma$ or $1$, the inertia ellipsoid of the planar rhombus central configuration
is not round. 
\end{lemma}

\section*{Thanks} to Mickael Gastineau and Herv\'e Manche for their help, at an early stage of this work, to understand numerically the discriminant, which forced me to look for another way of taming the set of symmetric matrices with a double eigenvalue; to Jacques F\'ejoz for computations discussions and his continued interest in the fourth dimension; to Alain Albouy for his remarks, as pitiless as illuminating, to Hugo Jim\'enez-P\'erez and Lei Zhao for their careful reading; 
and last but not least,
to Jacques Laskar, to whom this paper is dedicated. As I already said, it is his discovery using TRIP, of the fact that, contrarily to my first guess, crossings are not avoided at the linear level near the regular tetrahedron, which gave me the impetus to prove the main result of this paper.


\begin{thebibliography}{B}

\bibitem[A]{A} V.I. Arnold {\sl Mathematical Methods of Classical Mechanics} $2^{nd}$ edition, Springer (1989) 

\bibitem[AC]{AC} A. Albouy, A. Chenciner {\sl Le Probl\`eme des $N$ corps et les distances
mu\-tu\-elles}, Inventiones mathematicae 131 (1998), 151-184.

\bibitem[C0]{C0} A. Chenciner {\sl The ``form" of a triangle}, Rendiconti di Matematica, S\'erie VII, vol. 27, 1-16 (2007)

\bibitem[C1]{C1} A. Chenciner {\sl The angular momentum of a relative equilibrium}, Discrete and Continuous Dynamical Systems (num\'ero d\'edi\'e \`a Ernesto Lacomba)  (2012), Volume 33, Number 3, March 2013

\noindent {\small \texttt http://arxiv.org/abs/1102.0025} (2010).

\bibitem[C2]{C2} A. Chenciner {\sl The Lagrange reduction of the $N$-body problem: a survey}, 
Acta Mathematica Vietnamica (2013) 38: 165-186,\, 

\noindent {\small \texttt  http://arxiv.org/abs/1111.1334}

\bibitem[C3]{C3} A. Chenciner {\sl Symmetric balanced configurations of four equal masses}, manuscript

\bibitem[CJ]{CJ} A. Chenciner \& H. Jim\'enez-P\'erez {\sl Angular momentum and Horn's problem},  
Moscow Mathematical Journal, Vol. 13, Number 4, Oct--Dec. 2013, 621--630

\noindent {\small \texttt http://arxiv.org/abs/1110.5030} (2011).

\bibitem[D]{D} M. Domokos {\sl Discriminant of symmetric matrices as a sum of squares and the orthogonal group}
Communications Pure App. Math. Vol. 64, number 4, 443-465 (2011)

\noindent {\small \texttt  http://arxiv.org/abs/1003.0475} 

\bibitem[FFLP]{FFLP}  S. Fomin, W. Fulton, C.K. Li, Y.T. Poon 
{\sl Eigenvalues, singular values, and Littlewood-Richardson coefficients},
Amer. J. Math. {\bf 127}, no. 1, 101--127 (2005)

\bibitem[GL]{GL} M. Gastineau and J. Laskar,
{\sl TRIP: A Computer Algebra System Dedicated to Celestial Mechanics and Perturbation Series},
ACM Commun. Comput. Algebra, vol. 44, pp. 194--197, 2011,

\bibitem[La]{La} J. L. Lagrange {\sl Essai sur le probl\`eme des trois corps}, 1772

\bibitem[HZ]{HZ} Gert Heckman \& Lei Zhao {\sl Angular Momenta of Relative Equilibrium Motions and Real Moment Map Geometry}, preprint (2015),
{\small \texttt  http://arxiv.org/abs/1505.07331}

\bibitem[L]{L} P. D. Lax {\sl On the Discriminant of Real Symmetric Matrices}, Communications in Pure and Applied Mathematics, Vol. LI, 1387-1396 (1998)

\end{thebibliography}
\end{document}